\documentclass[a4paper,12pt]{article}
\usepackage[utf8]{inputenc}
\usepackage[T1]{fontenc}
\usepackage{amsmath, amssymb, amsthm, mathrsfs, stmaryrd, faktor, enumitem, tikz-cd, url, todonotes}
\usepackage[top=2.3cm,bottom=2cm,right=2cm,left=2cm]{geometry}

\numberwithin{equation}{section}

\newtheorem{theo}[equation]{Theorem}
\newtheorem{prop}[equation]{Proposition}
\newtheorem{lem}[equation]{Lemma}
\newtheorem{cor}[equation]{Corollary}

\theoremstyle{definition}
\newtheorem{defi}[equation]{Definition}
\newtheorem{ex}[equation]{Example}

\theoremstyle{remark}
\newtheorem{rem}[equation]{Remark}

\newcommand*{\Z}{\ensuremath\mathbb{Z}} 
\newcommand*{\Q}{\ensuremath\mathbb{Q}} 
\newcommand*{\R}{\ensuremath\mathbb{R}} 
\newcommand*{\F}{\ensuremath\mathbb{F}} 
\newcommand*{\set}[1]{\left\{#1\right\}} 
\newcommand*{\gltw}{\gl_2^{tw}} 
\newcommand*{\glzbis}[1]{\gl_2(#1)\times \gl_1(#1)} 
\newcommand*{\secs}[1]{\Gamma(#1,\os)} 
\newcommand*{\cc}{\ensuremath\mathcal{C}}
\newcommand*{\mm}{\ensuremath\mathcal{M}}
\newcommand*{\nn}{\ensuremath\mathcal{N}}
\newcommand*{\LL}{\ensuremath\mathcal{L}}
\newcommand*{\vv}{\ensuremath\mathcal{V}}

\newcommand*{\os}{\ensuremath\mathcal{O}_S}
\newcommand*{\ou}{\ensuremath\mathcal{O}_U}
\newcommand*{\vvu}{\ensuremath\mathcal{V}_{\vert U}}
\newcommand*{\LLU}{\ensuremath\mathcal{L}_{\vert U}}
\newcommand*{\ccu}{\ensuremath\mathcal{C}_{\vert U}}
\newcommand*{\nnu}{\ensuremath\mathcal{N}_{\vert U}}
\newcommand*{\ortau}{\ensuremath\overline{\tau}}

\newcommand*{\sheafhom}{\mathscr{H}\kern -.5pt om}

\DeclareMathOperator{\SL}{SL}
\DeclareMathOperator{\gl}{GL}
\DeclareMathOperator{\spec}{Spec}

\DeclareMathOperator{\Frac}{Frac}

\DeclareMathOperator{\Aut}{Aut}
\DeclareMathOperator{\sym}{Sym}
\DeclareMathOperator{\pic}{Pic}
\DeclareMathOperator{\tr}{Tr}

\DeclareMathOperator{\charac}{char}

\mathchardef\mhyphen="2D

\title{Recovering the Picard group of quadratic algebras from Wood's binary quadratic forms}
\author{William Dallaporta}
\date{April 2024}

\begin{document}

\maketitle

\begin{sloppypar}

\begin{small}\textit{Keywords:} Picard group, quadratic algebra, binary quadratic form\end{small}

\begin{small}\textit{MSC classes:} 11E16, 11R29, 14C22 (Primary)\end{small}

\paragraph{Abstract} Let $S$ be a scheme such that $2$ is not a zero divisor. In this paper, we address the following question: given a quadratic algebra over $S$, how can we parametrize its Picard group in terms of quadratic forms? In 2011, Wood established a set-theoretical bijection between isomorphism classes of primary binary quadratic forms over $S$ and isomorphism classes of pairs $(\cc,\mm)$ where $\cc$ is a quadratic algebra over $S$ and $\mm$ is an invertible $\cc$-module. Unexpectedly, examples suggest that a refinement of Wood's bijection is needed in order to parametrize Picard groups. This is why we start by classifying quadratic algebras over $S$; this is achieved by using two invariants, the discriminant and the parity. Extending the notion of orientation of quadratic algebras to the non-free case is another key step, eventually leading us to the desired parametrization. All along the paper, we illustrate various notions and obstructions with a wide range of examples.

\section{Introduction}

\subsection{Historical background and motivations}

The richness of the theory of integral binary quadratic forms has been proven many times since the work of Gauss in his \textit{Disquisitiones Arithmeticae} (1801). By \textit{integral binary quadratic form}, we mean a map ${q \colon \Z^2 \longrightarrow \Z}$ such that for all $x,y \in \Z$, we have ${q(x,y) = ax^2+bxy+cy^2}$ for some $a,b,c \in \Z$. Such a map is said to be \textit{primitive} if ${\gcd(a,b,c) = 1}$, and its \textit{discriminant} is the quantity ${\Delta = b^2-4ac}$. We have a natural action of $\gl_2(\Z)$ and $\SL_2(\Z)$ on integral quadratic forms, defined by ${\mu \cdot q := q \circ \mu}$ where $\mu \in \gl_2(\Z)$ or $\SL_2(\Z)$.

First studied in the framework of diophantine equations, quadratic forms appeared shortly after to have strong connections with the ideal class group of quadratic orders over $\Z$. The induced parametrization of ideal classes by classes of primitive quadratic forms is still one of the best tools to derive information about the ideal class group, such as its size \cite[Chapter~5]{Cohen}.

The correspondence is deep and useful, nevertheless it requires a certain amount of work to handle it properly. Gauss' computations are particularly hard to follow, and lots of mathematicians tried to simplify his arguments, making them clearer and at the same time more likely to be generalized (\textit{cf.} for instance \cite[p.~58 and 74]{Cox2}, or \cite[p.~32, historical note]{Towber}). A natural question is to replace the base ring $\Z$ by another one, such that we still have nice results on the associated quadratic forms, especially regarding the group structure and its link with ideal classes. One important feature of that problem is to determine which quadratic algebra corresponds to a given quadratic form. Over $\Z$, the discriminant of an integral quadratic form fully determines the corresponding quadratic order, but this is no longer true over more general rings, as Example~\ref{cexpar} shows.

The main motivation for the present note is to find a parametrization of the Picard group of a given quadratic algebra by a well-chosen set of classes of quadratic forms, over a general base scheme $S$ instead of $\Z$. To reach this goal, we shall start by establishing a classification of quadratic algebras, making clear which quadratic algebra corresponds to a given quadratic form.
\vspace{.5cm}

In 1980, Towber made a thorough review of previous works extending the base ring to get a group structure on quadratic forms \cite[introduction]{Towber}, with an extensive bibliography. To generalize a step further, he enriched the definition of quadratic form over a ring $R$, considering ${q \colon P \longrightarrow R}$ where $P$ is an oriented locally free rank $2$ $R$-module such that $\Lambda^2P$ is free of rank $1$ over $R$, instead of merely ${P = R^2}$ \cite[Definition~1.8]{Towber}. Therein, the orientation is nothing but the choice of a generator of $\Lambda^2P$. This definition of quadratic forms enabled Towber to handle sheaves of quadratic forms (from page 85 of \cite{Towber}) by using localization arguments.

As before, we say that an oriented binary quadratic form is \textit{primitive} if the ideal generated by the image of $q$ is the unit ideal. Towber considered the $\SL_2$-action, under which the orientation and the discriminant are preserved. In order to endow the set of $\SL_2$-classes of primitive oriented quadratic forms of given discriminant with a group structure, Towber introduced a new concept, the \textit{parity} of a quadratic form, which for an integral binary quadratic form ${ax^2+bxy+cy^2}$ is the class of $b$ modulo $2$. When $2$ is not a zero divisor in the base ring $R$, this allowed Towber to define a composition law on the set of $\SL_2$-classes of primitive oriented quadratic forms of given \textit{type}, that is, of given discriminant and parity. He showed directly that it is a group law, hence his work is independent from any connection with some Picard group.
\vspace{.5cm}

Almost at the same time (1982), Kneser (\cite{Kneser}) managed to obtain a group structure by selecting the quadratic forms whose associated even Clifford algebra $C$ is given. He defined a composition law on their $\gl_2(R)$-classes, without any condition on the base ring $R$. Then he linked it to the Picard group of $C$, but the obtained homomorphism is neither injective nor surjective in general. To recover an isomorphism with the full Picard group, he added the information of a second $R$-module, so that what he called \textit{quadratic maps} are triples $(M,N,q)$ where ${q \colon M \longrightarrow N}$ with $M$ locally free of rank $2$ and $N$ locally free of rank $1$, instead of ${N = R}$ as previously. Requiring $M$ to be a projective rank $1$ $C$-module and $q$ to be compatible with the norm of $C$, Kneser obtained the desired isomorphism with the Picard group of $C$.
\vspace{.5cm}

We will not use the even Clifford algebra point of view in this paper. However, the definition of quadratic maps was Wood's starting point for the generalisation of the theory to an arbitrary scheme in place of our base ring $R$ (\cite[2011]{Wood}). If $(S,\os)$ is a scheme, then her \textit{linear binary quadratic forms} are triples $(\vv,\LL,q)$ where $\vv$ and $\LL$ are locally free $\os$-modules of ranks $2$ and $1$ respectively. Here, $q$ is not seen as a map but rather as a global section of $\sym^2\vv \otimes_{\os} \LL$, which is the dual point of view. To define equivalence classes, Wood considered the action of $\gl_2(\vv)\times\gl_1(\LL)$, where the $\gl_2$-part acts on $\vv$ and the $\gl_1$-one on $\LL$ by change of bases. She established a bijection (\cite[Theorem~1.4]{Wood}) associating $\gl_2 \times \gl_1$-classes of linear binary quadratic forms on the one hand to classes of pairs $(\cc,\mm)$ on the other hand, where $\cc$ is a locally free $\os$-algebra of rank $2$ and $\mm$ is a traceable $\cc$-module (cf. \cite{Wood} for the definition of traceable). Restricting to primitive forms corresponds to requiring $\mm$ to be an invertible $\cc$-module in that bijection \cite[Theorem~1.5]{Wood}. Her motivations being related to moduli problems, she gave neither a group law nor an isomorphism with some Picard group, but rather a set-theoretical bijection with a disjoint union of quotient sets of Picard groups (in the primitive case).
\vspace{.5cm}

\subsection{Main ideas and results of this article}

The spirit of this paper is to recover from Wood's bijection an explicit bijection between the Picard group of a given quadratic algebra and a well-chosen set of classes of primitive quadratic forms over a general scheme $S$. Although it may seem straightforward, it is actually not obvious how to properly extract it from Wood's bijection. A natural approach is to restrict Wood's bijection to the case of primitive quadratic forms of fixed discriminant. But this leads to at least four problems, illustrated by Examples~\ref{cexpar} (the corresponding quadratic algebra is not uniquely determined), \ref{pasuni} ($2$ must be a non-zero-divisor), \ref{exzauto} (automorphisms may identify ideal classes) and \ref{exdelta} (the discriminant must be oriented).

Thus, our unique assumption on $S$ is that $2$ must not be a zero divisor. Our global strategy is first to determine which linear binary quadratic forms match to which quadratic algebras. Since automorphisms of quadratic algebras are responsible for the identification of some Picard classes, we then rigidify quadratic algebras with an adequate notion of orientation. This way, we remove the identification problem, but we create extra copies of the Picard group, parametrized by ``twisted'' quadratic forms. Classifying oriented quadratic algebras enables us to select exactly one copy.

The first well-known invariant to determine which quadratic forms correspond to which quadratic algebras is the discriminant. As previously said, over $\Z$ it is a complete invariant for quadratic orders, but is not enough in general to characterize a quadratic algebra (Example~\ref{cexpar}). Inspired by Towber's notion of parity, we introduce a scheme-theoretic version of it, for both quadratic algebras and linear binary quadratic forms (Definitions~\ref{defpargen} and \ref{defpargenlbqf}). Both the discriminant $\Delta$ and the parity $\Pi$ rely on some $\os$-module $\nn \in \pic(S)$, that is, $\Delta$ and $\Pi$ are global sections of $\nn^{\otimes 2}$ and $\nn/2\nn$ respectively. The triple $(\Delta,\Pi,\nn)$ constitutes what we call the \textit{type} of a quadratic algebra or of a linear binary quadratic form. We characterize all triples $(\Delta,\Pi,\nn)$ with ${\Delta \in \Gamma(S,\nn^{\otimes 2})}$ and ${\Pi \in \Gamma(S,\nn/2\nn)}$ appearing as the type of some quadratic algebra (Proposition~\ref{constr}). Our main result regarding classification of quadratic algebras by their types is Theorem~\ref{uniqueness}, which can be summarized as:
\begin{theo} \label{goalclass}
	When $2$ is not a zero divisor, the isomorphism class of a quadratic algebra is fully determined by the isomorphism class of its type.
\end{theo}

In the case when $S = \spec(R)$ and $\nn$ is free, one can extract this classification from the proof of \cite[Theorem~4.3]{Voightlibre}. Though not explicitly mentioned, the parity shows up in Voight's arguments. In another direction, he considered the situation when $2$ is a zero divisor, giving a partial answer to the classification by introducing the \textit{Artin-Schreier group}.

Coming back to quadratic forms, we then show that Wood's bijection is type-preserving (Theorem~\ref{Woodgl2gl1}), making possible the selection of the equivalence classes of linear binary quadratic forms corresponding to a given quadratic algebra.

The remaining obstruction to reach the Picard group in that bijection is the existence of non-trivial automorphisms of the corresponding quadratic algebra. Indeed, such automorphisms may identify for instance the Picard class of an invertible module with its inverse, as Wood already noticed over $\Z$ (introduction of \cite{Wood}). More generally, the following diagram summarizes Wood's bijection in the primitive case, with the use of the parity, as given in Theorem~\ref{Woodgl2gl1}: 
\begin{equation}\label{diagrgl2gl1}
	\begin{tikzcd}
		\begin{minipage}{.44\linewidth}
			$\faktor{\left\{
				\begin{array}{c}
					primitive~linear~binary \\
					quadratic~forms \\
					of~type~isomorphic~to \\
					(\Delta,\Pi,\nn)
				\end{array}
				\right\}}{\gl_2\times\gl_1}$
		\end{minipage}
		\arrow[r,leftrightarrow,"1~:~1"]
		& ~~\faktor{\pic(\cc_0)}{\sim}
	\end{tikzcd}
\end{equation}
\noindent where $\cc_0$ is a given quadratic algebra of type $(\Delta, \Pi,\nn)$, and $\sim$ is an equivalence relation defined in Theorem~\ref{Woodgl2gl1}. Our main idea to solve the problem is to introduce an $\nn$-\textit{orientation} of our quadratic algebra $\cc$, that is, to choose an isomorphism of $\cc/\os$ (or equivalently of $\Lambda^2\cc$) with a fixed representative $\nn^{\vee}$ of the isomorphism class of $\cc/\os$ (Definition~\ref{defori}). When $2$ is not a zero divisor, the only remaining automorphism is the identity, as desired.

On the quadratic forms' side, this amounts to a particular choice of the locally free module $\LL$, namely ${\LL = \Lambda^2\vv^{\vee} \otimes_{\os} \nn}$. Wood already studied different choices of $\LL$, and the closest to ours is ${\LL = \Lambda^2\vv^{\vee}}$, corresponding to the case when $\cc/\os$ is free (\cite[Theorem~5.2]{Wood}). The extra factor $\nn$ enables us to treat the general case. Inspired by Wood's terminology, we call \textit{$\nn$-twisted quadratic forms} the quadratic forms corresponding to our choice of $\LL$. This leads to the \textit{twisted action of the linear group}, which we denote by $\gl_2^{tw}$ (Definition~\ref{defgltw}).

For ${\nn \in \pic(S)}$, considering $\nn$-twisted quadratic forms and $\nn$-oriented quadratic algebras removes the problem of identification of Picard classes, but has the side effect of duplicating them. By Theorems~\ref{Woodgltw} and \ref{bij}, Diagram \eqref{diagrgl2gl1} becomes
\begin{center}
	\begin{tikzcd}
		\begin{minipage}{.36\linewidth}
			$\faktor{\left\{
				\begin{array}{c}
					primitive~\nn \mhyphen twisted \\
					quadratic~forms \\
					of~type~isomorphic~to \\
					(\Delta,\Pi,\nn)
				\end{array}
				\right\}}{\gltw}$
		\end{minipage}
		\arrow[r,leftrightarrow,"1~:~1"]
		& ~~\underset{i \in \kappa}{\bigsqcup} \pic(\cc_0)
	\end{tikzcd}
\end{center}
where $\kappa$ denotes the (possibly infinite) cardinality of the set of different pairs $(\varepsilon^2\Delta,\varepsilon\Pi)$ for $\varepsilon$ varying in $\Gamma(S,\os)^{\times}$. To reach the Picard group of $\cc_0$, we must refine the definition of type in that context. Using specific representatives of the discriminant and the parity, which we refer to as \textit{oriented} or \textit{natural} discriminants and parities, we prove that there is a unique oriented quadratic algebra of given oriented type, up to isomorphism (Theorem~\ref{oruniqueness}). Our main result, using definitions detailed after, is Theorem~\ref{bij}, which can be stated as follows:

\begin{theo} \label{goal}
	Let $S$ be a scheme such that $2$ is not a zero divisor. Let $\nn \in \pic(S)$. Then the set of $\gltw$-classes of primitive $\nn$-twisted binary quadratic forms over $S$ of fixed natural discriminant $\delta$ and natural parity $\pi$ is in bijection with the Picard group of the unique (up to isomorphism) $\nn$-oriented quadratic algebra $\cc_0$ of the same oriented discriminant $\Delta$ and oriented parity $\pi$.
\end{theo}

\begin{center}
	\begin{tikzcd}
		\begin{minipage}{.36\linewidth}
			$\faktor{\left\{
				\begin{array}{c}
					primitive~\nn \mhyphen twisted \\
					quadratic~forms~of \\
					natural~type~ (\delta,\pi,\nn)
				\end{array}
				\right\}}{\gltw}$
		\end{minipage}
		\arrow[r,leftrightarrow,"1~:~1"]
		& ~~\pic(\cc_0)
	\end{tikzcd}
\end{center}

Notice that the bijection can be made explicit, as in Corollary~\ref{expl} for instance.

\vspace{.5cm}

When $2$ is a zero divisor, one may wonder what kind of definitions or invariants one should consider to preserve such a parametrization of the Picard group. In the affine case, a partial answer is given by the aforementioned Artin-Schreier group \cite{Voightlibre}. On another side, still in the affine case, one can check carefully that Wood's construction of the quadratic algebra attached to a linear binary quadratic form $(\vv,\LL,q)$ corresponds to the even Clifford algebra of $q$, described by Kneser. Without any assumption on the base ring $R$, Kneser managed to derive a bijection with the Picard group of a given quadratic algebra over $R$ by fixing the even Clifford algebra, but the definitions and actions were different; in particular, the group acting on his quadratic maps was $\gl_2(R)$.

\vspace{.5cm}

Our notations are chosen to be as close as possible to Wood's in \cite{Wood}. Parts of the present article can be seen as a complement to \cite{Wood}, and would be more easily understandable with it nearby.

The organization of this article is as follows: in Subsection~\ref{secdefi}, we recall Wood's definitions we will work with, especially the discriminant. In Subsection~\ref{secpar}, we introduce the parity. When $2$ is not a zero divisor, we prove the main results regarding the classification of quadratic algebras in Subsection~\ref{secuni}, in particular Theorem~\ref{goalclass}. Then, we recall Wood's definition of linear binary quadratic forms and we relate them to quadratic algebras in Subsection~\ref{secgl2gl1}. Next, Subsection~\ref{sectwori} is devoted to the notion of orientation, together with twisted quadratic forms. In Subsection~\ref{sectwdiscpar}, we define the oriented and natural discriminants and parities. Finally, Subsection~\ref{sectwpic} completes the proof of Theorem~\ref{goal} by proving the uniqueness of an oriented quadratic algebra of given oriented type, when $2$ is not a zero divisor, and the resulting correspondence with the Picard group of this quadratic algebra.

\paragraph{Notations} Throughout this paper, we let $S$ be a scheme, and we denote by $\os$ its structure sheaf. It is unital and associative, as are all other rings and schemes we shall consider.

The Picard group of $S$ is denoted by $\pic(S)$. When we assume that $2$ is not a zero divisor on $S$, we mean that $2$ is not a zero divisor in $\Gamma(U,\os)$ for any open subset $U \subseteq S$.

Given a sheaf of $\os$-modules $\mm$, we write $\mm^{\vee}$ for the sheaf $\sheafhom(\mm,\os)$ of homomorphisms of $\os$-modules from $\mm$ to $\os$. If $U$ is an open subset of $S$, and if $\mm$ is a sheaf of rings or modules, we denote by $\Gamma(U,\mm)$ the ring or $\secs{U}$-module of sections of $\mm$ over $U$.

Let $R$ be a ring, and let ${r_1,\ldots,r_m \in R}$. We denote by ${\left\langle r_1,\ldots,r_m \right\rangle}$ the ideal generated by ${r_1,\ldots,r_m}$.

Let $R$ be a ring and let $M$ be a free $R$-module. If $(x_1,\ldots,x_n)$ is an $R$-basis of $M$, then $(x_1^{\vee},\ldots,x_n^{\vee})$ is its dual basis, and in particular a basis of $M^{\vee}$.

\paragraph{Acknowledgements.}

This work is part of my PhD at the Institut de Mathématiques de Toulouse. I particularly thank Jean Gillibert and Marc Perret for their support all along this work, both mathematically and on a personal level. I am grateful to them and to the anonymous referee for their careful reading and for the numerous improvements they suggested.

\section{Classification of quadratic algebras when $2$ is not a zero divisor}\label{classification}

\subsection{Definitions}\label{secdefi}

All the definitions of this Subsection are extracted from \cite{Wood}.

\begin{defi}
	A \textit{quadratic algebra over $S$} is a locally free $\os$-algebra of rank $2$.
\end{defi}

\begin{ex} \label{exfree}
	Over a scheme $S$, the sheaf ${\cc = \faktor{\os[\tau]}{\left\langle \tau^2+r\tau+s \right\rangle}}$ where ${r,s \in \Gamma(S,\os)}$ is a free quadratic algebra.
\end{ex}

\begin{ex} \label{exnotfree}
	Let $R$ be an integral domain, let $I$ be an invertible ideal of $R$. Let ${\rho, \sigma \in R}$ and define
	$$C(I,\rho,\sigma) := R \oplus xI$$
	as a subring of $\faktor{R[x]}{\left\langle x^2 + \rho x + \sigma \right\rangle}$. Then $C(I,\rho,\sigma)$ is a quadratic algebra, since there exists an open covering $(U_i)_i$ of $\spec(R)$ such that $I_{\vert U_i}$ is free of rank $1$, hence $C(I,\rho,\sigma)_{\vert U_i}$ is free of rank $2$. If $\alpha_i$ denotes a generator of $I_{\vert U_i}$, then
	$$C(I,\rho,\sigma)_{\vert U_i} = R_{\vert U_i} \oplus x\alpha_i R_{\vert U_i} \simeq \faktor{R_{\vert U_i}[\tau_i]}{\left\langle \tau_i^2 + \rho \alpha_i \tau_i + \sigma \alpha_i^2 \right\rangle}.$$
	
	More generally, one can take ${\rho \in I^{-1}}$ and ${\sigma \in I^{-2}}$ to get a richer family of examples.
\end{ex}

\begin{prop} \label{freedirectsummand}
	Let $\cc$ be a quadratic algebra over $S$. Then $\cc/\os$ is a locally free $\os$-module of rank $1$. Moreover, if $\cc$ is free as an $\os$-module, then there exists ${\tau \in \Gamma(S,\cc)}$ such that ${\cc \simeq \os \oplus \tau \os}$ as $\os$-modules. In particular, ${\cc/\os \simeq \tau \os}$ is free.
\end{prop}

\begin{proof}
	The fact that $\cc/\os$ is locally free of rank $1$ follows from the second part of the statement, whose proof is a straightforward generalization of \cite[Lemma~3.2]{Voightlibre}. Let us detail it here.
		
	If $\cc$ is free, then there exist ${\alpha, \beta \in \Gamma(S,\cc)}$ such that ${\cc \simeq \alpha \os \oplus \beta \os}$ as $\os$-modules. This implies that there exist ${r,s,t,u,v,w \in \Gamma(S,\os)}$ such that ${1 = r\alpha + s\beta}$, ${\alpha^2 = t\alpha + u\beta}$ and ${\alpha\beta = v\alpha + w\beta}$.
	
	Multiplying the first equation by $\alpha$, we get
	$$\alpha = r \alpha^2 + s\alpha\beta = (rt + sv)\alpha + (ru + sw)\beta.$$
	Since $(\alpha,\beta)$ is an $\os$-basis of $\cc$, we infer that ${1 = rt+sv}$. Therefore, the matrix $\begin{pmatrix} r & s \\ -v & t \end{pmatrix}$ has determinant $1$, and
	$$\begin{pmatrix} r & s \\ -v & t \end{pmatrix} \cdot \begin{pmatrix} \alpha \\ \beta \end{pmatrix} = \begin{pmatrix} 1 \\ \tau \end{pmatrix}$$
	where ${\tau := -v\alpha + t\beta}$. Thus, $(1,\tau)$ is another $\os$-basis, hence ${\alpha \os \oplus \beta \os = \os \oplus \tau \os}$.
\end{proof}

\begin{rem} \label{remcohom}
	In the affine case, it is true that a quadratic algebra $C$ over a ring $R$ is isomorphic to ${R \oplus C/R}$ as $R$-modules. This can be proved with Nakayama's Lemma (\cite[Lemma~1.3]{Voightcsuros}).
	
	 However, it does not extend to schemes, as Example~\ref{expassplitte} shows. The obstruction can be viewed as the fact that the following exact sequence
	 $$0 \longrightarrow \os \longrightarrow \cc \longrightarrow \cc/\os \longrightarrow 0$$
	 induces the long exact sequence
	 $$0 \longrightarrow H^0(S,\os) \longrightarrow H^0(S,\cc) \longrightarrow H^0(S,\cc/\os) \longrightarrow H^1(S,\os) \longrightarrow \ldots$$
	 In general, there is no reason for $H^1(S,\os)$ to be trivial, although it does when $S$ is an affine scheme, by the vanishing of sheaf cohomology in the affine case \cite[Lemma~5.1.9.2]{EGA1}.
\end{rem}

\begin{ex} \label{expassplitte}
	Let $E$ be an elliptic curve over $\F_2$ with a $2$-torsion point $P$ (for example the curve defined by the equation ${y^2+xy = x^3+x^2+1}$ with the point ${P=(0,1)}$). The multiplication-by-$2$ map ${[2] \colon E \longrightarrow E}$ is a finite morphism of degree 4, and it induces the multiplication-by-$2$ map on ${H^1(E, \mathcal{O}_E) \simeq \F_2}$, which is zero in our setting. It follows that the quartic algebra ${\mathcal{O}_E \longrightarrow [2]_*\mathcal{O}_E}$ has no linear splitting. This was pointed out by Bhatt \cite[Example~2.11]{Bhatt}.
	
	Quotienting out $E$ by $\left\langle P \right\rangle$ yields a $2$-isogeny ${\varphi \colon E \longrightarrow E' := E/\left\langle P \right\rangle}$. It has been well-known \cite[Theorem~III.6.1]{Silverman} that the dual isogeny ${\varphi^{\vee} \colon E' \longrightarrow E}$ satisfies ${\varphi^{\vee} \circ \varphi = [2]}$. Hence, multiplication by $2$ on the cohomology factors as
	\begin{center}
		\begin{tikzcd}
			H^1(E,\mathcal{O}_E) \arrow[r,"H^1\varphi"] & H^1(E',\mathcal{O}_{E'}) \arrow[r,"H^1\varphi^{\vee}"] & H^1(E,\mathcal{O}_E)
		\end{tikzcd}
	\end{center}
	where all $H^1$ are $1$-dimensional $\F_2$-vector spaces. The composite map being zero, at least one of the maps is zero. By the same reasoning as above, one of the corresponding quadratic algebras has no linear section.
\end{ex}

Proposition~\ref{freedirectsummand} implies that every free quadratic algebra $\cc$ over a scheme $S$ can be written as in Example~\ref{exfree}. A change of basis on such an algebra corresponds to another choice of generator $\tau'$, which modifies the defining equation of $\cc$. We will often use it so we make it explicit: if ${\tau' = \varepsilon \tau + \alpha}$ for some ${\varepsilon \in \Gamma(S,\os)^{\times}}$ and ${\alpha \in \Gamma(S,\os)}$, then the equation ${\tau^2+r\tau+s = 0}$ becomes
\begin{equation} \label{eq}
	\tau'^2 + (r\varepsilon - 2\alpha)\tau' + \alpha^2-r\alpha\varepsilon + s\varepsilon^2 = 0.
\end{equation}

Moreover, if $U \subseteq S$ is an open subset such that $\ccu$ is free of rank $2$, then we can choose an $\ou$-basis of the shape $(1,\tau)$ for some quadratic section ${\tau \in \Gamma(U,\cc)}$. It induces the $\ou$-basis $(\overline{\tau})$ of $(\cc/\os)_{\vert U}$, where $\overline{\tau}$ is the image of $\tau$ in $(\cc/\os)_{\vert U}$.

\begin{defi} \label{deltaw}
	The \textit{discriminant of a quadratic algebra $\cc$} is the pair $(\Delta(\cc),(\cc/\os)^{\vee})$ where $\Delta(\cc)$ is the global section of $((\cc/\os)^{\vee})^{\otimes 2}$ locally defined as follows. Let $U \subseteq S$ be some open subset such that $\ccu$ is free; take an $\ou$-basis $(1,\tau)$ where ${\tau^2 + r\tau + s = 0}$ for some ${r,s \in \secs{U}}$, then ${\Delta(\cc)_{\vert U} := (r^2-4s)(\ortau^{\vee}\otimes \ortau^{\vee})}$, with $\ortau$ the image of $\tau$ in $(\cc/\os)_{\vert U}$. By abuse of language, we often refer to $\Delta(\cc)$ as the discriminant of $\cc$.
	
	A discriminant is a pair $(\Delta,\nn)$ where $\nn$ is a locally free $\os$-module of rank $1$ and $\Delta $ is a global section of $\nn^{\otimes 2}$. We say that two discriminants $(\Delta,\nn)$ and $(\Delta',\nn')$ are \textit{isomorphic} if there exists an isomorphism of $\os$-modules ${f \colon \nn \overset{\sim}{\longrightarrow} \nn'}$ such that ${\Delta' = f^{\otimes 2}(\Delta)}$.
\end{defi}

\begin{rem}
	Let ${r,s \in \Gamma(S,\os)}$ and $\varepsilon \in \Gamma(S,\os)^{\times}$. Notice that the two quadratic algebras $\faktor{\os[\tau]}{\left\langle \tau^2+r\tau+s \right\rangle}$ and $\faktor{\os[\tau]}{\left\langle \tau^2+\varepsilon r\tau+\varepsilon^2 s \right\rangle}$ are isomorphic through the ring homomorphism sending $\tau$ to $\varepsilon^{-1} \tau$, but the discriminant is multiplied by $\varepsilon^2$ when passing from the first to the second one. Therefore, we must allow flexibility on discriminants, so that we can multiply them by the square of any unit. This is why we shall deal with isomorphism classes of discriminants instead of the discriminants themselves. Over $\Z$, this does not make a difference, but for more general rings or schemes, we must take care of this.
\end{rem}

\begin{prop} \label{discwood}
	The discriminant $\Delta$ of a quadratic algebra is well-defined, and two isomorphic quadratic algebras have isomorphic discriminants.
\end{prop}

This Proposition was already pointed by Wood, but as other results will be proved in a similar way in this paper, we give a detailed proof.

\begin{proof}
	Let $\cc$ be a quadratic algebra. Assume first that $\cc$ is free. By Proposition~\ref{freedirectsummand}, we can write ${\cc = \faktor{\os[\tau]}{\left\langle \tau^2+r\tau+s \right\rangle}}$ for some ${r,s \in \Gamma(S,\os)}$. If ${\tau' := \varepsilon \tau + \alpha}$ for some ${\varepsilon \in \Gamma(S,\os)^{\times}}$ and ${\alpha \in \Gamma(S,\os)}$, then Equation (\ref{eq}) enables us to compute that
	\begin{align*}
		\Delta(\cc) & = ((r\varepsilon-2\alpha)^2 - 4(\alpha^2-r\alpha\varepsilon + s\varepsilon^2))\overline{\tau'}^{\vee}\otimes \overline{\tau'}^{\vee} \\
		& = (r^2\varepsilon^2 - 4s\varepsilon^2)(\varepsilon^{-1}\ortau^{\vee} \otimes \varepsilon^{-1}\ortau^{\vee}) = (r^2-4s)\ortau^{\vee}\otimes \ortau^{\vee},
	\end{align*}
	which proves that the discriminant does not depend on the choice of $\tau$, hence is well-defined. In the general case, the above reasoning is valid over $\ccu$ for every open subset $U \subseteq S$ such that $\ccu$ is free. Taking an open covering, we must check the compatibility on overlaps, but this essentially amounts to a change of basis, a case we have just treated. Thus, $\Delta$ is globally well-defined.
	
	Finally, given an isomorphism of $\os$-algebras ${\psi \colon \cc \overset{\sim}{\longrightarrow} \cc'}$, we need to check that the discriminants of $\cc$ and $\cc'$ are isomorphic. Let $U \subseteq S$ be some open subset such that $\ccu$ is free. If $(1,\tau)$ is an $\ou$-basis of $\ccu$, then $\ccu'$ has $\ou$-basis $(1,\psi(\tau))$. In particular, $\tau$ and $\psi(\tau)$ satisfy the same equation, leading to the same local computation of the discriminant. We infer that $\psi$ induces an isomorphism of $\os$-modules ${\overline{\psi} \colon \cc/\os \overset{\sim}{\longrightarrow} \cc'/\os}$ such that ${\Delta(\cc) = \overline{\psi}^{\vee\otimes 2}(\Delta(\cc'))}$, as desired.
\end{proof}

\begin{ex}
	Let us consider again the quadratic algebra ${C = C(I,\rho,\sigma)}$ from Example~\ref{exnotfree}. Then $\Delta(C)$ is a global section of $((xI)^{\vee})^{\otimes 2}$, which is locally given on $U_i$ by $\alpha_i^2(\rho^2-4\sigma)\overline{\tau_i}^{\vee}\otimes\overline{\tau_i}^{\vee}$. One can check that the family $(\alpha_i\overline{\tau_i}^{\vee})_i$ is compatible on overlaps; indeed, on $U_i \cap U_j$, both $\alpha_i$ and $\alpha_j$ are generators of $I_{\vert U_i \cap U_j}$, hence there exists ${\varepsilon_{i,j} \in R_{\vert U_i\cap U_j}^{\times}}$ such that ${\alpha_j = \varepsilon_{i,j} \alpha_i}$. We also have ${\overline{\tau_j} = x\alpha_j = x\varepsilon_{i,j}\alpha_i = \varepsilon_{i,j} \overline{\tau_i}}$, hence ${\overline{\tau_j}^{\vee} = \varepsilon_{i,j}^{-1}\overline{\tau_i}^{\vee}}$, leading to ${\alpha_j\overline{\tau_j}^{\vee} = \alpha_i\overline{\tau_i}^{\vee}}$ as desired. Therefore, there exists a linear form ${\varphi \in (xI)^{\vee}}$ such that ${\Delta(C) = (\rho^2-4\sigma)\varphi\otimes\varphi}$.
\end{ex}

Over $\Z$, it has been well-known that isomorphism classes of quadratic orders are parametrized by the discriminant. However, the following example shows that it is not necessarily the case over more general base rings.

\begin{ex}\label{cexpar}
	Let ${R = \Z[\sqrt{8}]}$, let ${\Delta \in R}$ be a nonsquare element such that $4$ divides $\Delta$. Then the two quadratic algebras
	$$C_1 := \faktor{R[\tau_1]}{\langle \tau_1^2-\frac{\Delta}{4} \rangle} \text{~~~and~~~} C_2 := \faktor{R[\tau_2]}{\langle \tau_2^2+\sqrt{8}\tau_2-\frac{\Delta-8}{4} \rangle}$$
	have isomorphic discriminants ${\Delta(C_1) = \Delta \overline{\tau_1}^{\vee}\otimes\overline{\tau_1}^{\vee}}$ and ${\Delta(C_2) = \Delta \overline{\tau_2}^{\vee}\otimes\overline{\tau_2}^{\vee}}$, whereas they are not isomorphic; the obstruction being the fact that $2$ does not divide $\sqrt{8}$ in $R$.
\end{ex}

\subsection{The parity of a quadratic algebra} \label{secpar}

The notion of parity of a quadratic form seems to have appeared first in Towber's work (\cite{Towber}). This enabled him to get rid from a previous condition that the base ring $R$ should satisfy to preserve a group structure on quadratic forms, which can be formulated as the implication ${x^2 \equiv y^2 \pmod{4R} \Longrightarrow x \equiv y \pmod{2R}}$. Clearly, this condition is not satisfied in Example~\ref{cexpar}, since ${\sqrt{8}^2 \equiv 2^2 \pmod{4\Z[\sqrt{8}]}}$, but ${\sqrt{8} \not\equiv 2 \pmod{2\Z[\sqrt{8}]}}$. The parity of a quadratic form consists in the parity of its middle coefficient; it is a quite simple invariant which, surprisingly, is enough to solve the problem of classification of quadratic algebras, at least when $2$ is not a zero divisor (an assumption we shall not make in this Subsection). In \cite[Theorem~4.3]{Voightlibre}, Voight implicitly used the parity to classify quadratic algebras $\cc$ over $S = \spec(R)$ such that $\cc/\os$ is free.

Working over a general scheme $S$, there may be a small issue with the reduction modulo $2$. For any open subset ${U \subseteq S}$ and for any $\os$-module $\nn$, we denote by ${\chi_{U,\nn} \colon \faktor{\Gamma(U,\nn)}{2\Gamma(U,\nn)} \longrightarrow \Gamma(U,\nn/2\nn)}$ the canonical homomorphism of rings. Then $\chi_{U,\nn}$ is always injective, but when $U$ is not affine, it may fail to be surjective. Concretely, the parity will be locally defined in the image of this homomorphism.

\begin{defi} \label{defpargen}
	Let $\cc$ be a quadratic algebra. The \textit{parity} of $\cc$ is the pair $(\Pi(\cc),(\cc/\os)^{\vee})$ where $\Pi(\cc)$ is the global section of $\faktor{(\cc/\os)^{\vee}}{2(\cc/\os)^{\vee}}$ locally defined as follows. If ${U \subseteq S}$ is an open subset such that $\ccu$ is free of rank $2$, then ${\ccu \simeq \faktor{\ou[\tau]}{\left\langle \tau^2+r\tau+s \right\rangle}}$ for some ${r,s \in \Gamma(U,\os)}$, and we define $\Pi(\cc)_{\vert U}$ to be the image under $\chi_{U,(\cc/\os)^{\vee}}$ of the class of $r \ortau^{\vee}$ modulo $2\Gamma(U,(\cc/\os)^{\vee})$. Most of the time, we refer to $\Pi(\cc)$ as the parity of $\cc$, by abuse of language.
	
	As for the discriminant, a parity is of the form $(\Pi,\nn)$, with $\nn$ a locally free $\os$-module of rank $1$ and $\Pi$ a global section of $\nn/2\nn$. We say that two parities $(\Pi,\nn)$ and $(\Pi',\nn')$ are \textit{isomorphic} if there exists an isomorphism of $\os$-modules ${f \colon \nn \overset{\sim}{\longrightarrow} \nn'}$ such that ${\overline{f}(\Pi) = \Pi'}$, where ${\overline{f} \colon \nn/2\nn \longrightarrow \nn'/2\nn'}$ is induced by $f$.
\end{defi}

\begin{prop} \label{parinv}
	The parity of a quadratic algebra is well-defined, and two isomorphic quadratic algebras have isomorphic parities.
\end{prop}

\begin{proof}
	The proof follows the same arguments as the one of Proposition~\ref{discwood}. In particular, if ${\cc = \faktor{\os[\tau]}{\left\langle \tau^2+r\tau+s \right\rangle}}$ for some ${r,s \in \Gamma(S,\os)}$, and if ${\tau' := \varepsilon \tau + \alpha}$ for some ${\varepsilon \in \Gamma(S,\os)^{\times}}$ and ${\alpha \in \Gamma(S,\os)}$, then we have ${\Pi(\cc) \equiv (r\varepsilon - 2\alpha)\overline{\tau'}^{\vee} \equiv r\ortau^{\vee} \pmod{2\Gamma(S,\os)}}$.
\end{proof}

\begin{ex}
	Continuing with ${C = C(I,\rho,\sigma)}$ from Example~\ref{exnotfree}, for all $i$ we have ${\Pi(C)_{\vert U_i} \equiv \rho\alpha_i\overline{\tau_i}^{\vee} \pmod{2(xI_{\vert U_i})^{\vee}}}$. In that case, since $\rho$ is globally defined, one can check that we have a global lift of the parity given by ${\Pi(C) \equiv \rho \varphi \pmod{2(xI)^{\vee}}}$ where ${\varphi \in (xI)^{\vee}}$ is the glueing of the $\alpha_i\overline{\tau_i}^{\vee}$'s.
\end{ex}

\begin{rem}
	Proposition~\ref{parinv} fixes the problem encountered in Example~\ref{cexpar}. Although the two non-isomorphic quadratic algebras $C_1$ and $C_2$ have isomorphic discriminants, they have non-isomorphic parities since ${\Pi(C_1) \equiv 0}$ and ${\Pi(C_2) \equiv \sqrt{8}\overline{\tau_2}^{\vee} \not\equiv 0 \pmod{2\Z[\sqrt{8}]^{\vee}}}$.
\end{rem}

\subsection{Discriminant and parity as a complete invariant when $2$ is not a zero divisor} \label{secuni}

The goal of this Subsection is to prove Theorem~\ref{goalclass}.

\begin{defi} \label{deftype}
	The \textit{type of a quadratic algebra $\cc$} is the triple $(\Delta(\cc),\Pi(\cc),(\cc/\os)^{\vee})$.
	
	A type is a triple $(\Delta,\Pi,\nn)$ with $\nn$ a locally free $\os$-module of rank $1$, $\Delta$ a global section of $\nn^{\otimes 2}$ and $\Pi$ a global section of $\nn/2\nn$. We say that two types $(\Delta,\Pi,\nn)$ and $(\Delta',\Pi',\nn')$ are \textit{isomorphic} if there exists an isomorphism of $\os$-modules ${f \colon \nn \overset{\sim}{\longrightarrow} \nn'}$ such that ${f^{\otimes 2}(\Delta) = \Delta'}$ and ${\bar{f}(\Pi) = \Pi'}$, where ${\bar{f} \colon \nn/2\nn \longrightarrow \nn'/2\nn'}$ is induced by $f$.
	
	A triple $(\Delta,\Pi,\nn)$ is \textit{valid} if it is isomorphic to the type of some quadratic algebra.
\end{defi}

We shall give a characterization of valid triples independent from quadratic algebras. This will be done in Proposition~\ref{constr}.

\begin{prop} \label{typeinv}
	Two isomorphic quadratic algebras have isomorphic types.
\end{prop}

\begin{proof}
	Let $\psi \colon \cc \overset{\sim}{\longrightarrow} \cc'$ be an isomorphism of quadratic algebras. Following the proofs of Propositions~\ref{discwood} and \ref{parinv}, one can check that $\psi$ induces an isomorphism of $\os$-modules ${(\cc/\os)^{\vee} \overset{\sim}{\longrightarrow} (\cc'/\os)^{\vee}}$ sending $(\Delta(\cc),\Pi(\cc))$ to $(\Delta(\cc'),\Pi(\cc'))$.
\end{proof}

The remaining part of this Subsection is devoted to showing that two quadratic algebras having isomorphic types are themselves isomorphic, when $2$ is not a zero divisor. This will be done in several steps.

\begin{lem}\label{rel}
	Let $(\Delta,\Pi,\nn)$ be a valid triple. Then there exists an open covering $(U_i)_i$ of $S$ such that for all $i$,
	\begin{enumerate}
		\item $\nn_{\vert U_i}$ is free of rank $1$;
		
		\item the section ${\Pi_{\vert U_i} \in \Gamma(U_i,\nn/2\nn)}$ can be lifted to some ${\tilde{\Pi}_i \in \Gamma(U_i,\nn)}$;
		\item for any such $\tilde{\Pi}_i$, we have the congruence relation ${\Delta_{\vert U_i} \equiv \tilde{\Pi}_i^{\otimes 2} \pmod{4\Gamma(U_i,\nn^{\otimes 2})}}$.
	\end{enumerate}
\end{lem}

\begin{proof}
	Since the triple $(\Delta,\Pi,\nn)$ is valid, it comes from a quadratic algebra $\cc$. First, assume that $\cc$ is free. Then ${\cc = \faktor{\os[\tau]}{\left\langle \tau^2 + r\tau + s \right\rangle}}$ for some ${r,s \in \Gamma(S,\os)}$ and ${\tau \in \Gamma(S,\cc)}$, as seen in Proposition~\ref{freedirectsummand}. In particular, ${\cc/\os = \ortau\os \simeq \nn^{\vee}}$, hence $\nn$ is free. Moreover, $\Pi$ coincides with the class of $r\ortau^{\vee}$ modulo $2\Gamma(S,(\cc/\os)^{\vee})$, by definition of the parity of $\cc$. Hence, ${\tilde{\Pi} := r\ortau^{\vee}}$ is a lift of $\Pi$ to $\Gamma(S,(\cc/\os)^{\vee})$.
	
	The definition of the discriminant ${\Delta(\cc) = \Delta}$ enables us to conclude that $$\Delta = (r^2-4s)\ortau^{\vee}\otimes\ortau^{\vee} \equiv r^2\ortau^{\vee}\otimes\ortau^{\vee} = \tilde{\Pi}^{\otimes 2} \pmod{4\Gamma(S,(\cc/\os)^{\vee\otimes 2})}.$$
	Thus, when $\cc$ is free, we proved the result with $S$ being the sought open covering of itself. In the general case, there exists an open covering $(U_i)_i$ of $S$ such that $\cc_{\vert U_i}$ is free of rank $2$ for all $i$. The previous arguments show that this open covering satisfies the desired conditions.
\end{proof}

\begin{rem}
	The fact that $\nn$ is free does not necessarily imply that $\cc$ is free: as seen in Remark~\ref{remcohom}, a potential obstruction to this lies in $H^1(S,\os)$. In particular, the open covering $(U_i)_i$ of Lemma~\ref{rel} may be finer than the one expected to satisfy the weaker condition that $\nn_{\vert U_i}$ is free.
\end{rem}

Let us focus on the free case.

\begin{lem} \label{freeok}
	Let $S$ be a scheme such that $2$ is not a zero divisor and let ${\cc = \faktor{\os[\tau]}{\left\langle\tau^2+r\tau+s\right\rangle}}$ be a free quadratic algebra over $S$. Let ${d := r^2-4s}$ and ${\tilde{p} \in \Gamma(S,\os)}$ be such that ${\tilde{p} \equiv r \pmod{2\Gamma(S,\os)}}$.
	
	Then $\cc$ is isomorphic to ${\Omega := \faktor{\os[\omega]}{\left\langle \omega^2+\tilde{p}\omega - \frac{d - \tilde{p}^2}{4}\right\rangle}}$ through the isomorphism of sheaves of rings
	$$\psi \colon \left\{
	\begin{aligned}
		\cc & \longrightarrow \Omega \\
		\tau & \longmapsto \omega + \frac{\tilde{p} - r}{2}
	\end{aligned}\right..$$
\end{lem}

\begin{proof}
	The fact that $2$ is not a zero divisor ensures us that if $2$ divides some ${x \in \Gamma(S,\os)}$, then there is a unique ${\alpha \in \Gamma(S,\os)}$ such that ${2\alpha = x}$. Thus, the quantity $\frac{d-\tilde{p}^2}{4}$ makes sense in $\Gamma(S,\os)$, and we denote it by $\alpha$. Likewise, we write $\beta := \frac{\tilde{p}-r}{2}$.
	
	We check that $\psi$ is a homomorphism of rings, that is, that ${\psi(\tau^2) = \psi(\tau)^2}$. We have
	$$\psi(\tau^2) = -r\psi(\tau) -s = -r\omega - s - r\beta,$$
	whereas
	$$\psi(\tau)^2 = -\tilde{p}\omega + \alpha +2\beta\omega + \beta^2 = (-\tilde{p}+2\beta)\omega + (\alpha + \beta^2).$$
	Furthermore, ${2\beta = \tilde{p}-r}$, hence ${r = \tilde{p}-2\beta}$ and the $\omega$-terms in the above two equations are equal. For the constant terms, we compute their difference:
	\begin{align*}
		\alpha + \beta^2 + r\beta + s & = \frac{d-\tilde{p}^2}{4} + \frac{(\tilde{p}-r)^2}{4} + r\frac{\tilde{p}-r}{2} + s \\
		& = \frac{d-2\tilde{p}r + r^2 + 2\tilde{p}r - 2r^2+4s}{4} \\
		& = \frac{d-r^2+4s}{4} = 0.
	\end{align*}
	Therefore, ${\psi(\tau^2) = \psi(\tau)^2}$, which proves that $\psi$ is a homomorphism of sheaves of rings. Clearly, $\psi$ is invertible, with inverse map ${\omega \mapsto \tau - \frac{\tilde{p}-r}{2}}$, hence $\psi$ is an isomorphism.
\end{proof}

\begin{rem}
	The isomorphism $\psi$ from Lemma~\ref{freeok} is a generalization of the following observation. Over $\R$, consider two polynomial equations ${x^2+rx+s = 0}$ and ${y^2+r'y+s' = 0}$ with the same (usual) discriminant ${\Delta \geq 0}$, and denote by ${x_{\pm} := \frac{-r \pm \sqrt{\Delta}}{2}}$ and ${y_{\pm} := \frac{-r' \pm \sqrt{\Delta}}{2}}$ their respective roots. Then, we have ${x_+ = y_+ + \frac{r'-r}{2}}$.
\end{rem}

The next Proposition characterizes valid types $(\Delta,\Pi,\nn)$, and if a valid type is given, it constructs a representative of the isomorphism class of quadratic algebras of that type. By the way, it also gives a converse to Lemma~\ref{rel}.

\begin{prop} \label{constr}
	Let $S$ be a scheme such that $2$ is not a zero divisor. Let $\nn$ be a locally free $\os$-module of rank $1$, let $\Delta$ be a global section of $\nn^{\otimes 2}$ and let $\Pi$ be a global section of $\nn/2\nn$. Then
	\begin{enumerate}
		\item the triple $(\Delta,\Pi,\nn)$ is valid if and only if there exists an open covering $(U_i)_i$ such that for all $i$:
		\begin{enumerate}
			\item $\nn_{\vert U_i} = \nu_i\mathcal{O}_{U_i}$ is free; \label{cond1}
			
			\item the section $\Pi_{\vert U_i} \in \Gamma(U_i,\nn/2\nn)$ can be lifted to some $\tilde{\Pi}_i \in \Gamma(U_i,\nn)$; \label{cond2}
			\item for any such lift $\tilde{\Pi}_i$, we have $\Delta_{\vert U_i} \equiv \tilde{\Pi}_i^{\otimes 2} \pmod{4\Gamma(U_i,\nn^{\otimes 2})}$. \label{cond3}
		\end{enumerate}
		\item If $(\Delta,\Pi,\nn)$ satisfies conditions \eqref{cond1}, \eqref{cond2} and \eqref{cond3}, then one obtains a quadratic algebra $\Omega$ over $S$ of type $(\Delta,\Pi,\nn)$ by glueing of the $\Omega_i$'s, where ${\Omega_i := \mathcal{O}_{U_i} \oplus \nu_i^{\vee}\mathcal{O}_{U_i}}$ is a free $\mathcal{O}_{U_i}$-module whose algebra structure is defined by ${\omega_i^2+\tilde{p}_i\omega_i-\frac{d_i-\tilde{p}_i^2}{4} = 0}$ with $1 := (1,0)$,  $\omega_i := (0,\nu_i^{\vee})$, and ${\tilde{p}_i,d_i \in \Gamma(U_i,\os)}$ are given by ${\tilde{\Pi}_i = \tilde{p}_i\nu_i}$, ${\Delta_{\vert U_i} = d_i \nu_i \otimes \nu_i}$.
	\end{enumerate}
\end{prop}

\begin{proof}
	The implication ($(\Delta,\Pi,\nn)$ valid) $\Rightarrow$ (\eqref{cond1}, \eqref{cond2} and \eqref{cond3}) of the first assertion comes from Lemma~\ref{rel}, whereas the converse one follows from the second assertion about $\Omega$. Therefore, we focus on the latter.
	
	The $\Omega_i$'s are well-defined since $2$ is not a zero divisor, and because condition (\ref{cond3}) ensures that ${d_i^2 \equiv \tilde{p}_i^2 \pmod{4\Gamma(U_i,\os)}}$. In order to glue the $\Omega_i$'s together, we need to define isomorphisms ${\psi_{i,j} \colon \Omega_{i \vert U_i \cap U_j} \overset{\sim}{\longrightarrow} \Omega_{j \vert U_i \cap U_j}}$ for all $i,j$. To lighten the notation, if ${x_i \in \Gamma(U_i,\Omega_i)}$, we still denote by $x_i$ its restriction to $\Omega_{i \vert U_i \cap U_j}$. The free algebra $\Omega_{i \vert U_i \cap U_j}$ is generated by the $\mathcal{O}_{U_i \cap U_j}$-basis $(1,\omega_i)$ with ${\omega_i^2 + \tilde{p}_i \omega_i - \frac{d_i - \tilde{p}_i^2}{4} = 0}$, and $\Omega_{j \vert U_i \cap U_j}$ by $(1,\omega_j)$ with ${\omega_j^2 + \tilde{p}_j \omega_j - \frac{d_j - \tilde{p}_j^2}{4} = 0}$. On $U_i \cap U_j$, the free module $\nn_{\vert U_i \cap U_j}$ is both generated by $\nu_i$ and $\nu_j$, hence there exists $\varepsilon_{i,j} \in \Gamma(U_i \cap U_j, \os)^{\times}$ such that ${\nu_j = \varepsilon_{i,j}\nu_i}$. Since ${\overline{\omega_j} = \nu_j^{\vee}}$, we infer that ${\overline{\omega_j} = \varepsilon_{i,j}^{-1}\overline{\omega_i}}$. Likewise, ${\tilde{p}_j\nu_j = \tilde{p}_j\varepsilon_{i,j}\nu_i}$ and ${d_j\nu_j\otimes\nu_j = d_j\varepsilon_{i,j}^2\nu_i\otimes\nu_i}$, leading to ${\tilde{p}_i \equiv \tilde{p}_j\varepsilon_{i,j} \pmod{2\Gamma(U_i \cap U_j,\os)}}$ and ${d_i = d_j\varepsilon_{i,j}^2}$. Inspired by the $\psi$ from Lemma~\ref{freeok}, we set
	
	$$\psi_{i,j} \colon \left\{
	\begin{aligned}
		\Omega_{i \vert U_i \cap U_j} & \longrightarrow \Omega_{j \vert U_i \cap U_j} \\
		\omega_i & \longmapsto \varepsilon_{i,j}\left(\omega_j + \frac{\tilde{p}_j - \tilde{p}_i\varepsilon_{i,j}^{-1}}{2}\right)
	\end{aligned}\right..$$

	Following the proof of Lemma~\ref{freeok}, one can check that $\psi_{i,j}$ is an isomorphism of quadratic algebras.
	
	It remains to check the glueing condition for the $\psi_{i,j}$. Given three indices $i,j,k$, we must check that the following diagram
	\begin{center}
		\begin{tikzcd}[column sep=scriptsize]
			\Omega_{i \vert U_i \cap U_j \cap U_k} \arrow[rr,"\psi_{i,k}"] \arrow[rd,"\psi_{i,j}"'] & & \Omega_{k \vert U_i \cap U_j \cap U_k} \\ & \Omega_{j \vert U_i \cap U_j \cap U_k} \arrow[ru,"\psi_{j,k}"'] &
		\end{tikzcd}
	\end{center}
	commutes. We have $$\psi_{i,k}(\omega_i) = \varepsilon_{i,k}\left(\omega_k + \frac{\tilde{p}_k-\tilde{p}_i\varepsilon_{i,k}^{-1}}{2}\right),$$
	and it is straightforward to compute that
	$$\psi_{j,k} \circ \psi_{i,j}(\omega_i) = \varepsilon_{i,j}\varepsilon_{j,k}\left(\omega_k + \frac{\tilde{p}_k-\tilde{p}_i\varepsilon_{i,j}^{-1}\varepsilon_{j,k}^{-1}}{2}\right).$$
	
	By definition of the $\varepsilon$'s, on $U_i \cap U_j \cap U_k$ we have ${\nu_k = \varepsilon_{j,k} \nu_j = \varepsilon_{j,k}\varepsilon_{i,j}\nu_i}$, but at the same time ${\nu_k = \varepsilon_{i,k}\nu_i}$, hence ${\varepsilon_{i,k} = \varepsilon_{i,j}\varepsilon_{j,k}}$. Therefore, ${\psi_{i,k} = \psi_{j,k}\circ\psi_{i,j}}$ and the above diagram indeed commutes. By virtue of \cite[(3.3.1)]{EGA1}, we can glue all the $\Omega_i$'s together to obtain a sheaf of $\os$-algebras $\Omega$. By construction, $\Omega$ is a quadratic algebra over $S$ which has discriminant $\Delta$ and parity $\Pi$.
\end{proof}

\begin{rem}
	Actually, if we consider any open affine covering of $S$, then condition (\ref{cond2}) in Proposition~\ref{constr} is always satisfied, by the vanishing of sheaf cohomology in the affine case.
\end{rem}

We now turn to Theorem~\ref{goalclass}, which we state in detail.

\begin{theo} \label{uniqueness}
	Let $S$ be a scheme such that $2$ is not a zero divisor. Then we have the following bijection:
	\begin{align*}
		\left\{
		\begin{array}{c}
		isomorphism~classes \\
		of~valid~triples
		\end{array}
		\right\} & \overset{1 : 1}{\longleftrightarrow} \left\{
		\begin{array}{c}
		isomorphism~classes \\
		of~quadratic~algebras
		\end{array}
		\right\} \\
		[(\Delta,\Pi,\nn)] ~ & \longmapsto ~ \left[\Omega(\Delta,\Pi,\nn)\right] \\
		[(\Delta(\cc),\Pi(\cc),(\cc/\os)^{\vee})] ~ & \longmapsfrom ~ [\cc]
	\end{align*}
	where $\Omega(\Delta,\Pi,\nn)$ is the quadratic algebra constructed in Proposition~\ref{constr}.
\end{theo}

\begin{proof}
	We already saw that two isomorphic quadratic algebras $\cc$ and $\cc'$ have isomorphic types (Proposition~\ref{typeinv}). By definition of valid triples, the map sending quadratic algebras to their types is surjective. For injectivity, we shall show that any quadratic algebra of type isomorphic to $(\Delta,\Pi,\nn)$ is isomorphic to $\Omega(\Delta,\Pi,\nn)$.

	Let $(\Delta,\Pi,\nn)$ be a valid triple. Let $(U_i)_i$ be an open covering satisfying conditions \eqref{cond1}, \eqref{cond2} and \eqref{cond3} from Proposition~\ref{constr}, and let ${\Omega := \Omega(\Delta,\Pi,\nn)}$ be the quadratic algebra defined in this same Proposition. Then for all $i$, the restriction of $\Omega$ to $U_i$ is given by ${\Omega_i := \faktor{\mathcal{O}_{U_i}[\omega_i]}{\left\langle \omega_i^2+\tilde{p}_i\omega_i - \frac{d_i - \tilde{p}_i^2}{4}\right\rangle}}$, where ${\Delta_{\vert U_i} = d_i\overline{\omega_i}^{\vee}\otimes \overline{\omega_i}^{\vee}}$ and ${\Pi \equiv \tilde{p}_i\overline{\omega_i}^{\vee} \pmod{2\Gamma(U_i,\nn)}}$, with $\overline{\omega_i}$ the image of $\omega_i$ in ${\Omega_i/\mathcal{O}_{U_i}}$.
	
	Let $\cc$ be a quadratic algebra of type isomorphic to $(\Delta,\Pi,\nn)$. Note that if $(V_j)_j$ is an open covering of $S$ such that $\cc_{\vert V_j}$ is free for all $j$, then ${W_{i,j} := U_i \cap V_j}$ is an open subset such that $\cc_{\vert W_{i,j}}$ and $\nn_{\vert W_{i,j}}$ are free. In particular, the open covering $(W_{i,j})_{i,j}$ satisfies conditions \eqref{cond1}, \eqref{cond2} and \eqref{cond3} as well as $(U_i)_i$. Therefore, up to refining the covering $(U_i)_i$, we assume without loss of generality that $\cc_{\vert U_i}$ is free for all $i$.
	
	Let $(1,\tau_i)$ be an $\mathcal{O}_{U_i}$-basis of $\cc_{\vert U_i}$, where ${\tau_i^2 + r_i\tau_i + s_i = 0}$. Fix an isomorphism ${\theta \colon \cc/\os \overset{\sim}{\longrightarrow} \nn^{\vee}}$ sending the type of $\cc$ to $(\Delta,\Pi,\nn)$. Since both $\theta(\overline{\tau_i})$ and $\overline{\omega_i}$ are generators of $\nn_{\vert U_i}^{\vee}$, there must exist a unit $\varepsilon_i \in \Gamma(U_i,\os)^{\times}$ such that $\theta(\overline{\tau_i}) = \varepsilon_i\overline{\omega_i}$. Inspired by the previous isomorphisms of rings we defined in Lemma~\ref{freeok} and Proposition~\ref{constr}, let $\Psi_i$ be defined as follows:
	$$\Psi_i \colon \left\{
	\begin{aligned}
		\cc_{\vert U_i} & \longrightarrow \Omega_i \\
		\tau_i & \longmapsto \varepsilon_i\left(\omega_i + \frac{\tilde{p}_i - r_i\varepsilon_i^{-1}}{2}\right)
	\end{aligned}\right..$$
	Proceeding as in the proof of Lemma~\ref{freeok}, the verification that the map $\Psi_i$ is an isomorphism of quadratic algebras over $U_i$ ends with the equation ${d_i\varepsilon_i^2 - (r_i^2-4s_i) = 0}$, which is true since ${d_i = (r_i^2-4s_i)\overline{\omega_i}^{\vee}(\theta(\overline{\tau_i}))^{-2}}$.
	
	Now that the setting is clear, we prove that $\cc$ is globally isomorphic to $\Omega$ by glueing together the $\Psi_i$'s. Following \cite[(3.3.2)]{EGA1}, it is sufficient to check that for all $i,j$, the following diagram
	\begin{center}
		\begin{tikzcd}[row sep=scriptsize]
			& \Omega_{i \vert U_i \cap U_j} \arrow[dd,"\psi_{i,j}"] \\
			\cc_{\vert U_i \cap U_j} \arrow[ru,"\Psi_i"] \arrow[rd,"\Psi_j"'] & \\
			& \Omega_{j \vert U_i \cap U_j}
		\end{tikzcd}
	\end{center}
	commutes, where $\psi_{i,j}$ is the isomorphism constructed in the proof of Proposition~\ref{constr}, given by ${\psi_{i,j}(\omega_i) = \varepsilon_{i,j}\left(\omega_j + \frac{\tilde{p}_j - \tilde{p}_i\varepsilon_{i,j}^{-1}}{2}\right)}$ where $\varepsilon_{i,j}$ is a unit defined by ${\overline{\omega_i} = \varepsilon_{i,j}\overline{\omega_j}}$. We have
	$$\psi_{i,j} \circ \Psi_i(\tau_i) = \psi_{i,j}\left(\varepsilon_i\left(\omega_i + \frac{\tilde{p}_i - r_i\varepsilon_i^{-1}}{2}\right)\right) = \varepsilon_i\varepsilon_{i,j}\left(\omega_j + \frac{\tilde{p}_j - r_i\varepsilon_i^{-1}\varepsilon_{i,j}^{-1}}{2}\right),$$
	but to compute $\Psi_j(\tau_i)$ we need to express $\tau_i$ as a function of $\tau_j$. The units previously introduced enable us to establish the link: we have
	$$\theta(\overline{\tau_i}) = \varepsilon_i\overline{\omega_i} = \varepsilon_i\varepsilon_{i,j}\overline{\omega_j} = \varepsilon_i\varepsilon_{i,j}\varepsilon_j^{-1}\theta(\overline{\tau_j}),$$
	hence ${\overline{\tau_i} = \eta\overline{\tau_j}}$ with ${\eta := \varepsilon_i\varepsilon_{i,j}\varepsilon_j^{-1}}$, since $\theta$ is an isomorphism. Writing $\tau_i = \eta \tau_j + \alpha$ for some $\alpha \in \Gamma(U_i \cap U_j,\os)$ to determine, Equation (\ref{eq}) tells us that $r_i = r_j\eta-2\alpha$, hence $\alpha = \frac{r_j\eta-r_i}{2}$ (since $2$ is not a zero divisor by hypothesis), so that
	$$\tau_i = \varepsilon_i\varepsilon_{i,j}\varepsilon_j^{-1}\left(\tau_j + \frac{r_j - r_i\varepsilon_i^{-1}\varepsilon_{i,j}^{-1}\varepsilon_j}{2}\right).$$
	
	Thus, we can compute
	\begin{align*}
		\Psi_j(\tau_i) & = \Psi_j\left(\varepsilon_i\varepsilon_{i,j}\varepsilon_j^{-1}\left(\tau_j + \frac{r_j - r_i\varepsilon_i^{-1}\varepsilon_{i,j}^{-1}\varepsilon_j}{2}\right)\right) \\
		& = \varepsilon_i\varepsilon_{i,j}\left(\omega_j + \frac{\tilde{p}_j - r_j\varepsilon_j^{-1}}{2}\right) + \varepsilon_i\varepsilon_{i,j}\varepsilon_j^{-1}\frac{r_j - r_i\varepsilon_i^{-1}\varepsilon_{i,j}^{-1}\varepsilon_j}{2} \\
		& = \psi_{i,j} \circ \Psi_i(\tau_i),
	\end{align*}
	as desired. Therefore, the quadratic algebra $\cc$ is isomorphic to $\Omega$, concluding the proof.
\end{proof}

Given some discriminant, we may wonder when the parity is necessary to preserve the uniqueness of the corresponding quadratic algebra. Here is a partial answer.

\begin{prop} \label{paspar}
	Given a valid triple, its parity is completely determined by its discriminant at least in the three following cases:
	\begin{itemize}[label=-,topsep=0.1cm,parsep=0cm,itemsep=0.1cm]
		\item $S$ is a scheme such that $2$ is a unit;
		\item $S$ is a scheme such that $2$ generates a prime ideal sheaf;
		\item $S$ is a normal scheme.
	\end{itemize}
\end{prop}

\begin{proof}
	Let $(\Delta,\Pi,\nn)$ be a valid triple. If $2$ is a unit, then the quotient $\nn/2\nn$ is trivial, hence the parity must be $0$.
	
	For the two other cases, let $\Pi, \Pi' \in \Gamma(S,\nn/2\nn)$ be two parities, such that the triples $(\Delta,\Pi,\nn)$ and $(\Delta,\Pi',\nn)$ are valid. Let $(U_i)_i$ be an open covering of $S$ such that $\nn_{\vert U_i}$ is free. Fix $U$, one of the $U_i$'s. As seen in Proposition~\ref{constr}, if $\tilde{\Pi}_U$ and $\tilde{\Pi}'_U$ denote lifts of $\Pi_{\vert U}$ and $\Pi'_{\vert U}$ to $\Gamma(U,\nn)$, then ${\tilde{\Pi}_U^{\otimes 2} \equiv \Delta_{\vert U} \equiv \tilde{\Pi}'^{\otimes 2}_U \pmod{4\Gamma(U,\nn^{\otimes 2})}}$. Writing ${\Delta_{\vert U} = d \nu \otimes \nu}$, ${\tilde{\Pi}_U = \tilde{p}\nu}$ and ${\tilde{\Pi}'_U = \tilde{p}'\nu}$ for some ${d,\tilde{p},\tilde{p}' \in \Gamma(U,\os)}$ and for $\nu$ a generator of $\nnu$, we have the relation ${\tilde{p}^2 \equiv \tilde{p}'^2 \pmod{4\Gamma(U,\os)}}$.
	
	In particular, ${(\tilde{p}-\tilde{p}')^2 \equiv 0 \pmod{2\Gamma(U,\os)}}$, hence if $2$ generates a prime ideal sheaf in $\os$, we get ${\tilde{p} \equiv \tilde{p}' \pmod{2\Gamma(U,\os)}}$.
	
	Now, suppose that $S$ is normal. If $U$ is not affine, let $(V_j)_j$ be an open affine covering of $S$. For all $j$, there is an open covering of ${U \cap V_j}$ by open affine subsets $W_{j,k}$. Considering the $W_{j,k}$'s instead of $U$ enables to assume that $U$ is affine. Therefore, the ring $\Gamma(U,\os)$ is integrally closed, hence integral. On the other hand, we know that there exists ${s \in \Gamma(U,\os)}$ such that ${\tilde{p}^2 - \tilde{p}'^2 = 4s}$. If ${2 = 0}$ in $\ou$, then ${(\tilde{p} - \tilde{p}')^2 = 0}$, but $\Gamma(U,\os)$ is an integral domain, hence ${\tilde{p} = \tilde{p}'}$ and we are done. Otherwise, let ${t := \frac{\tilde{p} - \tilde{p}'}{2} \in \Frac(\secs{U})}$. A direct computation shows that we have the equation $$t^2 + \tilde{p}' t - s = 0,$$ hence ${t \in \Gamma(U,\os)}$ by normality, that is, ${\tilde{p} \equiv \tilde{p}' \pmod{2\Gamma(U,\os)}}$. Since $\Pi$ and $\Pi'$ coincide locally on an open covering of $S$, they must be equal.
\end{proof}

\begin{rem}
	The normality condition was already known by Butts and Estes (\cite[Corollary~2.6]{Butts-Estes}). Actually, one can check carefully that it is enough for $S$ to be normal over points of residual characteristic $2$.
\end{rem}

\begin{ex}
	Let ${R = \Z[\sqrt{8}]}$, then $R$ does not satisfy any of the conditions of Proposition~\ref{paspar} (one can check that ${\sqrt{2} \notin R_{\left\langle 2,\sqrt{8}\right\rangle} \subset \Q(\sqrt{2})}$). In fact, the parity is necessary in that case to distinguish two quadratic algebras with the same discriminant, in view of Example~\ref{cexpar}.
\end{ex}

In Theorem~\ref{uniqueness}, it is necessary to assume that $2$ is not a zero divisor in $\os$, otherwise the triple $(\Delta, \Pi,\nn)$ might not be enough to characterize isomorphism classes of quadratic algebras! We give below two examples of non-isomorphic free quadratic algebras with isomorphic types, one in characteristic $2$ with ${\Delta \neq 0}$ (Example~\ref{pasuni}), the other over $\Z/8\Z$ with ${\Delta = 0}$ (Example~\ref{pasuni2}).

\begin{ex} \label{pasuni}
	Let ${R = \F_2[X]/\langle X^2+X+1 \rangle \simeq \F_4}$ and let ${S = \spec(R)}$. To sense the problem, notice that since the characteristic of $R$ is $2$, the parity is an element of $R$ and the discriminant is its square, hence the information is redundant. Consider the two quadratic algebras defined by
	$$C_1 = \faktor{R[\tau_1]}{\langle \tau_1^2+\tau_1+1 \rangle} \simeq \F_4 \times \F_4 \text{~~~and~~~} C_2 = \faktor{R[\tau_2]}{\langle \tau_2^2+\tau_2+X \rangle} \simeq \F_8.$$
	Both $C_1$ and $C_2$ have types isomorphic to $(1,1,R)$, but they are not isomorphic.
\end{ex}

\begin{ex} \label{pasuni2}
	Let ${R = \Z/4\Z}$ and ${C_s = \faktor{R[\tau]}{\left\langle \tau^2-s \right\rangle}}$ for ${s \in R}$. Then the four quadratic algebras $C_s$ all have type $(0,0,R)$, but one can check that they are not isomorphic.
\end{ex}

\subsection{Linear binary quadratic forms and Wood's bijection} \label{secgl2gl1}

We now focus on linear binary quadratic forms. Defining the parity in that context will complete the bridge between them and quadratic algebras.

All definitions, except those about the parity and the type, are extracted from \cite{Wood}.

\begin{defi} \label{lbqf}
	A \textit{linear binary quadratic form over $S$} is a triple $(\vv,\LL,q)$ where $\vv$ is a locally free $\os$-module of rank $2$, $\LL$ is a locally free $\os$-module of rank $1$, and $q$ is a global section of $\sym^2\vv\otimes_{\os} \LL$.
	
	Let $U \subseteq S$ be an open subset such that $\vvu$ and $\LLU$ are free. Let $(x,y)$ be an $\ou$-basis of $\vvu$ and $(z)$ be an $\ou$-basis of $\LLU$. Omitting the tensor products by abuse of notation, $q$ can then be expressed locally on $U$ as ${ax^2z+bxyz+cy^2z}$ for some ${a,b,c \in \Gamma(U,\os)}$, which we sometimes denote by $[a,b,c]$, omitting the underlying bases. Then $(\vv,\LL,q)$ is said to be \textit{primitive} if, for every such $U$, the ideal generated by $a,b,c$ in $\secs{U}$ is the whole ring $\secs{U}$.
\end{defi}

\begin{defi}
	Two linear binary quadratic forms $(\vv,\LL,q)$ and $(\vv',\LL',q')$ are \textit{equivalent} if there exist two isomorphisms of $\os$-modules ${f \colon \vv \overset{\sim}{\longrightarrow} \vv'}$ and ${g \colon \LL \overset{\sim}{\longrightarrow} \LL'}$ such that $\sym^2f \otimes g$ sends $q$ to $q'$. This can be seen as a left action of ${\gl_2(\vv) \times \gl_1(\LL)}$, which we refer to as the ${\gl_2 \times \gl_1}$-action.
	
	Locally, we can describe this action explicitly. Let ${U \subseteq S}$ be some open subset such that $\vvu$ and $\LLU$ are free. Since every ${q \in \Gamma(U,\sym^2\vv \otimes_{\os} \LL)}$ is a linear combination of elementary tensors, it is enough to compute ${(\mu,\varepsilon)\cdot x_1\otimes x_2 \otimes z}$ for every $x_1,x_2 \in \Gamma(U, \vv)$ and every $z \in \Gamma(U,\LL)$. Then a pair ${\left(\mu,\varepsilon\right) \in \Gamma(U,\glzbis{\os})}$ acts on $x_1\otimes x_2 \otimes z$ via
	\begin{equation} \label{action}
		\left(\mu,\varepsilon\right)\cdot x_1\otimes x_2 \otimes z :=  \mu x_1\otimes \mu x_2 \otimes \varepsilon z.
	\end{equation}
\end{defi}

\begin{defi}
	The \textit{discriminant $\Delta$ of a linear binary quadratic form $(\vv,\LL,q)$} is the pair $(\Delta(\vv,\LL,q),\Lambda^2\vv\otimes_{\os} \LL)$ where $\Delta(\vv,\LL,q)$ is the global section of $(\Lambda^2\vv\otimes_{\os} \LL)^{\otimes 2}$ locally defined as follows. Let $U \subseteq S$ be some open subset such that ${\vvu \simeq \mathcal{O}_U^2}$ and ${\LLU \simeq \mathcal{O}_U}$; we can write ${q = ax^2z + bxyz + cy^2z}$ in some $\ou$-bases $(x,y)$ of $\vvu$ and $(z)$ of $\LLU$, and we set ${\Delta(\vv,\LL,q)_{\vert U} := (b^2-4ac) ((x \wedge y) \otimes z)^{\otimes 2}}$.
\end{defi}

\begin{rem} \label{pbdisc}
	The effect of some matrix $\mu$ in $\gl_2(\os)$ (acting on $\vv$) on the discriminant is to multiply it by $\det(\mu)^2$, while the $\gl_1$-part, acting on $\LL$, multiplies it by the square of some unit. Therefore, considering $\gl_2\times\gl_1$-classes necessitates weakening the usual invariance of the discriminant, as we did for quadratic algebras.
\end{rem}

As for Definition~\ref{defpargen}, we denote by ${\chi_{U,\nn} \colon \faktor{\Gamma(U,\nn)}{2\Gamma(U,\nn)} \hookrightarrow \Gamma(U,\nn/2\nn)}$ the canonical homomorphism of rings, for any open subset $U \subseteq S$ and $\os$-module $\nn$.

\begin{defi} \label{defpargenlbqf}
	The \textit{parity of a linear binary quadratic form $(\vv,\LL,q)$} is the pair $(\Pi(\vv,\LL,q),\Lambda^2\vv\otimes_{\os} \LL)$ where $\Pi(\vv,\LL,q)$ is the global section of $\faktor{\Lambda^2\vv\otimes_{\os} \LL}{2(\Lambda^2\vv\otimes_{\os} \LL)}$ locally defined as follows. Let ${U \subseteq S}$ be an open subset such that $\vvu$ and $\LLU$ are free. Let $(x,y)$ and $(z)$ be respective $\ou$-bases of $\vvu$ and $\LLU$, and write ${q = ax^2z+bxyz+cy^2z}$ for some ${a,b,c \in \Gamma(U,\os)}$. Then we define $\Pi(\vv,\LL,q)_{\vert U}$ to be the image under $\chi_{U,\Lambda^2\vv\otimes \LL}$ of the class of $b (x \wedge y) \otimes z$ modulo ${2\Gamma(U,\Lambda^2\vv\otimes_{\os} \LL)}$.
\end{defi}

\begin{defi}
	Let $(\vv,\LL,q)$ be a linear binary quadratic form. Then its \textit{type} is the triple $(\Delta(\vv,\LL,q),\Pi(\vv,\LL,q),\Lambda^2\vv\otimes_{\os} \LL)$.
	
	We say that two types are \textit{isomorphic} if they are in the sense of Definition~\ref{deftype}.
\end{defi}

One can prove the following Proposition by applying the same reasoning as in the context of quadratic algebras (Propositions~\ref{discwood} and \ref{parinv}). One can also derive it from the bijection of Theorem~\ref{Woodgl2gl1}.

\begin{prop} \label{typelbqf}
	The discriminant and the parity of a linear binary quadratic form are well-defined, and two equivalent linear binary quadratic forms have isomorphic types.
\end{prop}

One of the main features of the type is that Wood's bijection (\cite[Theorem~1.5]{Wood}) is type-preserving, meaning that if a linear binary quadratic form matches to a pair $(\cc,\mm)$, then they have isomorphic types. This is proved in the following Theorem, stating the obtained bijection when we fix the type and when $2$ is not a zero divisor on the base scheme $S$.

\begin{theo} \label{Woodgl2gl1}
	Let $S$ be a scheme such that $2$ is not a zero divisor, and let $\cc_0$ be a quadratic $\os$-algebra.
	For all $\mm$, ${\mm' \in \pic(\cc_0)}$, we write ${\mm \sim \mm'}$ if there exists an automorphism $\psi$ of $\cc_0$ over $S$ such that ${\mm' \simeq \psi^{\ast}\mm}$ as $\cc_0$-modules. Then there is a set-theoretical bijection
	$$ \left\{
	\begin{array}{c}
	\gl_2 \times \gl_1 \!{} \mhyphen equivalence~classes~of~primitive \\
	linear~binary~quadratic~forms~whose~type \\
	is~isomorphic~to~ (\Delta(\cc_0),\Pi(\cc_0),(\cc_0/\os)^{\vee})
	\end{array}
	\right\} \overset{1 : 1}{\longleftrightarrow} \faktor{\pic(\cc_0)}{\sim}$$
	
	By $\psi^{\ast}\mm$, we mean the $\cc_0$-module $\mm$ whose structure is locally given by ${\lambda \cdot x = \psi^{-1}(\lambda)x}$.
\end{theo}

\begin{proof}
	We start from Wood's Theorem 1.5 in \cite{Wood}. The bijection she established links equivalence classes of primitive linear binary quadratic forms and classes of pairs of the form $(\cc,\mm)$ where $\cc$ is a quadratic algebra and $\mm$ is an invertible $\cc$-module. Two such pairs $(\cc,\mm)$ and $(\cc',\mm')$ are isomorphic if there exists an isomorphism of quadratic algebras ${\psi \colon \cc \overset{\sim}{\longrightarrow} \cc'}$ and an isomorphism of $\cc'$-modules ${\varphi \colon \mm \otimes_{\cc} \cc' \overset{\sim}{\longrightarrow} \mm'}$.
	
	According to \cite[Remark~2.2]{Wood}, the local construction of a pair $(\cc,\mm)$ from a linear binary quadratic form $(\vv,\LL,q)$ in the bijection is as follows. Over an open subset $U \subseteq S$ such that $\vvu$ and $\LLU$ are free, let $(x,y)$ and $(z)$ be $\ou$-bases of $\vvu$ and $\LLU$ respectively, and let $(x^{\vee},y^{\vee})$ and $(z^{\vee})$ be the respective dual bases. Write $q = ax^2z+bxyz+cy^2z$ for some ${a,b,c \in \Gamma(U,\os)}$. Then
	$$\ccu = \ou \oplus (\Lambda^2\vvu\otimes_{\ou} \LLU)^{\vee} = \ou \oplus \tau \ou$$
	as $\ou$-modules where ${\tau := (x^{\vee} \wedge y^{\vee}) \otimes z^{\vee}}$, and the algebra structure is given by ${\tau^2+b\tau+ac = 0}$. Regarding $\mm$, one takes ${\mm = \vv}$ as $\os$-modules, and its $\ccu$-module structure is given by ${\tau \cdot x = -bx-cy}$ and ${\tau \cdot y = ax}$.
	
	Computing the discriminant and the parity, we have
	$$\Delta(\vv,\LL,q)_{\vert U} = (b^2-4ac)((x \wedge y) \otimes z))^{\otimes 2} = (b^2-4ac)\ortau^{\vee}\otimes\ortau^{\vee} = \Delta(\cc)_{\vert U}$$
	and
	$$\Pi(\vv,\LL,q)_{\vert U} \equiv b((x \wedge y) \otimes z) = b\ortau^{\vee} \equiv \Pi(\cc)_{\vert U} \pmod{2\Gamma(U,\Lambda^2\vvu\otimes_{\ou} \LLU)}.$$
	Therefore, Wood's bijection is type-preserving, in the sense that if a linear binary quadratic form matches to a pair $(\cc,\mm)$ in the bijection, then they have isomorphic types.
	
	When $2$ is not a zero divisor on $S$, Theorem \ref{uniqueness} applies and isomorphism classes of quadratic algebras are determined by their types. Thus, given a quadratic algebra $\cc_0$, restricting Wood's bijection to linear binary quadratic forms whose type is isomorphic to the one of $\cc_0$ gives the desired bijection. In particular, every class of pairs $(\cc,\mm)$ with $\cc$ having isomorphic type to the one of $\cc_0$ has a representative of the form $(\cc_0,\mm_0)$ for some $\mm_0$.
\end{proof}

\begin{rem}
	If we do not require our quadratic forms to be primitive, then one can check that we still have a bijection in Theorem~\ref{Woodgl2gl1}, but with classes of \textit{traceable} modules instead of invertible ones in the right-hand side of the bijection (\textit{cf}. \cite{Wood} for the definition of traceable).
\end{rem}

\begin{rem} \label{conj}
	One must take care about the fact that $\faktor{\pic(\cc)}{\sim}$ is not a group in general! For instance, over $\Z$, there is a unique non-trivial automorphism of quadratic algebras (called the conjugation) which identifies an ideal class with its inverse, as shown in the Example below\ldots
\end{rem}

\begin{ex} \label{exzauto}
	The class group of $\Z[\sqrt{-11}]$ is of size $3$, and representatives of the three different ideal classes are $\left\langle 1,\sqrt{-11} \right\rangle$, $I := \left\langle 3,\sqrt{-11}-1 \right\rangle$, $J := \left\langle 3,\sqrt{-11}+1 \right\rangle$. But the conjugation sends $I$ to $J$, hence Theorem~\ref{Woodgl2gl1} gives a bijection between two sets of two elements. More precisely, it associates to principal ideals the class of the linear binary quadratic form $x^2z+11y^2z$ on the one hand, and to the class of $I$ the class of the linear binary quadratic form $3x^2z+2xyz+4y^2z$ on the other hand.
\end{ex}

\section{The Picard group of a quadratic algebra when $2$ is not a zero divisor} \label{pic}

Given a quadratic algebra $\cc_0$, Theorem~\ref{Woodgl2gl1} is close to a parametrization of its Picard group, the only obstruction being the existence of non-trivial automorphisms of $\cc_0$. The goal of this Section is to rigidify quadratic algebras so that we remove all the non-trivial automorphisms, with the notion of \textit{orientation}. When $2$ is not a zero divisor, this will be enough to recover the Picard group in Theorem~\ref{Woodgl2gl1}.

\subsection{Orientation and twist} \label{sectwori}

As already noticed in Remark~\ref{conj}, quadratic algebras over $\Z$ already have a non-trivial automorphism, the conjugation. But over non-integral rings or schemes, there may be many more automorphisms!

\begin{ex} \label{exauto}
	Let ${R_1 = R_2 = \Z}$. Choose $d_1$ and $d_2 \in \Z$ non-squares, and let ${C_1 = \Z[\sqrt{d_1}]}$, ${C_2 = \Z[\sqrt{d_2}]}$. If ${\sigma_1 \in \Aut_{\Z}(C_1)}$ and ${\sigma_2 \in \Aut_{\Z}(C_2)}$ are the respective conjugations of $C_1$ and $C_2$, then the maps $id \times id$, $\sigma_1 \times id$, $id \times \sigma_2$ and $\sigma_1 \times \sigma_2$ are all automorphisms of the $\Z \times \Z$-quadratic algebra $C_1 \times C_2$.
\end{ex}

Wood already defined the orientation of a quadratic algebra $\cc$ in the case when $\cc/\os$ is free (\cite[Theorem~5.2]{Wood}). Inspired by her definition, we extend it to the general case.

\begin{defi} \label{defori}
	Let $\nn$ be a locally free $\os$-module of rank $1$. An \textit{$\nn$-oriented quadratic algebra} is a pair $(\cc,\theta)$ where $\cc$ is a quadratic algebra and ${\theta \colon \cc/\os \overset{\sim}{\longrightarrow} \nn^{\vee}}$ is an isomorphism of $\os$-modules, called an \textit{orientation} of $\cc$.
	
	Given two $\nn$-oriented quadratic algebras $(\cc,\theta)$ and $(\cc',\theta')$, an \textit{isomorphism of oriented quadratic algebras} is an isomorphism ${\psi \colon \cc \overset{\sim}{\longrightarrow} \cc'}$ of $\os$-algebras such that ${\theta = \theta' \circ \overline{\psi}}$, where ${\overline{\psi} \colon \cc/\os \longrightarrow \cc'/\os}$ is induced by $\psi$.
\end{defi}

\begin{rem}
	We have chosen the above convention (an isomorphism with the dual of $\nn$ instead of $\nn$ itself) in order to deal with discriminants and parities as global sections of $\nn^{\otimes 2}$ and $\nn/2\nn$ respectively.
\end{rem}

\begin{rem} \label{worifree}
	We recover Wood's notion of orientation (\cite[Theorem~5.2]{Wood}) by considering $\nn = \os^{\vee}$ in the above Definition.
\end{rem}

\begin{rem}
	Notice that ${\Lambda^2 \cc \simeq \cc/\os}$ through the isomorphism of $\os$-modules locally defined on open subsets such that $\cc$ is free by ${1\wedge\tau \mapsto \tau}$. Thus, an $\nn$-orientation may also be viewed as an isomorphism $\Lambda^2\cc \overset{\sim}{\longrightarrow} \nn^{\vee}$.
	
	When ${\cc = \faktor{\os[\tau]}{\left\langle \tau^2+r\tau+s \right\rangle}}$ is free, an $\os^{\vee}$-orientation of $\cc$ may be viewed as a choice of a classical orientation $1 \wedge \tau$ of $\cc$. In the general case, when $\cc/\os \simeq \nn^{\vee}$, an $\nn$-orientation can be interpreted as choices of local orientations of $\cc$ which are globally consistent.
\end{rem}

\begin{prop} \label{auto}
	Let $S$ be a scheme such that $2$ is not a zero divisor, and let $\nn \in \pic(S)$. Then the group of automorphisms of an $\nn$-oriented quadratic algebra over $S$ is trivial.
\end{prop}

\begin{proof}
	Let $(\cc,\theta)$ be an oriented quadratic algebra over $S$ and let $\psi$ be an automorphism of $(\cc,\theta)$. First, assume that $\cc$ is free. As seen in Proposition~\ref{freedirectsummand}, we can choose an $\os$-basis $(1,\tau)$ of $\cc$ such that $(\ortau)$ is a basis of $\cc/\os$, where $\tau \in \Gamma(S,\cc)$ is such that ${\tau^2 + r\tau + s = 0}$ for some ${r,s \in \secs{S}}$. Write ${\psi(\tau) = \varepsilon \tau + \alpha}$ for some ${\varepsilon \in \secs{S}^{\times}}$ and ${\alpha \in \secs{S}}$. Since $\psi$ is an automorphism, we must have ${\theta(\ortau) = \theta(\overline{\psi}(\ortau))}$, leading to ${\theta(\ortau) = \varepsilon \theta(\ortau)}$. We infer that ${\varepsilon = 1}$, for $\theta(\ortau)$ is a unit.
	
	Since $\psi$ is an automorphism, $\psi(\tau)$ must be a root of the polynomial ${X^2+rX+s}$. We compute: $$(\tau+\alpha)^2 + r(\tau+\alpha) + s = 2\alpha\tau + \alpha(\alpha + r),$$
	hence we must have ${2\alpha = 0}$ and ${\alpha(\alpha+r) = 0}$, by identification in the basis $(1,\tau)$. Since $2$ is not a zero divisor, we conclude that ${\alpha = 0}$, hence $\psi$ must be the identity map.
	
	In the general case, let $(U_i)_i$ be an open covering of $S$ such that $\cc_{\vert U_i}$ is free for all $i$. Then the previous reasoning is valid on $U_i$, hence $\psi$ coincides with the identity map on $U_i$ for all $i$. Therefore, $\psi$ must be the identity map globally.
\end{proof}

When $2$ is a zero divisor, non-trivial automorphisms can preserve the orientation, as Example~\ref{survieauto} shows in the free case.

\begin{ex} \label{survieauto}
	Let ${R = \Z/8\Z}$, let ${C := \faktor{R[\tau]}{\left\langle \tau^2-2 \right\rangle}}$. Then $\Aut_R(C)$ has cardinality $8$, and its elements are the maps ${\tau \mapsto u\tau+v}$, for ${u \in R^{\times}}$ and ${v \in \set{0,4}}$. Furthermore, if $\theta$ is an $R$-orientation of $C$, then there are two automorphisms of $(C, \theta)$, the identity and the map ${\tau \mapsto \tau + 4}$.
\end{ex}

\begin{ex}
	A simpler case is $R = \F_2$ and $C = \F_4$. For any $R$-orientation $\theta$ of $C$, the oriented quadratic algebra $(C,\theta)$ has two automorphisms.
\end{ex}

Theorem~\ref{Woodgl2gl1} links primitive linear binary quadratic forms to the data of a quadratic algebra and an invertible module. Taking $\os$-oriented quadratic algebras instead of non-oriented ones, Wood modified the notion of linear binary quadratic form to preserve the bijection. Doing it in the general case of $\nn$-oriented quadratic algebras for every $\nn \in \pic(S)$, we introduce the following notion.

\begin{defi} \label{defgltw}
	Let $\nn \in \pic(S)$. An \textit{$\nn$-twisted binary quadratic form}, denoted as a pair $(\vv,q)$, is the linear binary quadratic form obtained when we set ${\LL = \Lambda^2\vv^{\vee} \otimes_{\os} \nn}$ in Definition~\ref{lbqf}; in particular $q$ is a global section of ${\sym^2\vv \otimes_{\os} \Lambda^2 \vv^{\vee} \otimes_{\os} \nn}$. Locally on some open subset ${U \subseteq S}$ such that both $\vvu$ and $\nnu$ are free, $\ou$-bases $(x,y)$ of $\vvu$ and $(\nu)$ of $\nnu$ will induce the basis $(\nu\det_{(x,y)})$ on $\LLU$, where we write the tensor product as a product by abuse of notation. The induced left action is called the \textit{$\nn$-twisted action}, and we refer to it with the notation $\gltw$. Locally, a matrix ${\mu \in \gl_2(\ou)}$ acts in the twisted way on a quadratic form ${q = ax^2z + bxyz + cy^2z}$ (where ${z = \nu\det_{(x,y)}}$) as follows:
	\begin{equation} \label{actiontw}
		\mu \cdot q := \left(\mu,\frac{1}{\det(\mu)}\right) \cdot q,
	\end{equation}
	where the action on the right-hand side of (\ref{actiontw}) corresponds to the $\gl_2 \times \gl_1$-action from Equation (\ref{action}).
\end{defi}

\begin{rem} \label{fix}
	Notice that the extra factor $\frac{1}{\det(\mu)}$ in Equation (\ref{actiontw}) counterbalances the effect of $\gl_2$ on the discriminant, by Remark~\ref{pbdisc}.
\end{rem}

\begin{defi}
	Let $\nn \in \pic(S)$. An $\nn$-twisted binary quadratic form is said to be \textit{primitive} if it is when viewed as a linear binary quadratic form.
\end{defi}

Wood noticed that making specific choices of $\LL$ leads to other kinds of bijections, such as bijections with binary quadratic forms for ${\LL = \os}$ or twisted quadratic forms for ${\LL = \Lambda^2 \vv^{\vee}}$ (\cite[Theorems~5.1 and 5.2]{Wood}). In the same spirit, we state here the version corresponding to our setting (${\LL = \Lambda^2\vv^{\vee} \otimes_{\os} \nn}$), directly for primitive forms.

\begin{theo} \label{Woodgltw}
	Let $\nn \in \pic(S)$. There is a set-theoretical bijection
	$$ \left\{
	\begin{array}{c}
	\gltw \!{} \mhyphen equivalence~classes~of \\
	primitive~\nn \mhyphen twisted~binary \\
	quadratic~forms~over~ S
	\end{array}
	\right\} \overset{1 : 1}{\longleftrightarrow}  \left\{
	\begin{array}{c}
	isomorphism~classes~of~ (\cc, \theta, \mm), \\ with~ (\cc, \theta) ~an~\nn \mhyphen oriented \\ quadratic~algebra~over~ S, \\ ~and~ \mm ~an~invertible~ \cc \mhyphen module
	\end{array}
	\right\}.$$
	
	An isomorphism ${(\cc,\theta,\mm) \overset{\sim}{\longrightarrow} (\cc',\theta',\mm')}$ is a pair $(\psi,\varphi)$ where ${\psi \colon (\cc,\theta) \overset{\sim}{\longrightarrow} (\cc',\theta')}$ is an isomorphism of $\nn$-oriented algebras, and ${\varphi \colon \mm \otimes_{\cc} \cc' \overset{\sim}{\longrightarrow} \mm'}$ is an isomorphism of $\cc'$-modules.
\end{theo}

\subsection{Oriented and natural discriminant, oriented and natural parity} \label{sectwdiscpar}

We want to parametrize the Picard group of a given quadratic algebra $\cc_0$ using Theorem~\ref{Woodgltw} and the tools developed in Section~\ref{classification}. Assume that $2$ is not a zero divisor on the base scheme $S$, and let $\nn$ be a chosen representative of the isomorphism class of $(\cc_0/\os)^{\vee}$. in view of Proposition~\ref{auto}, the only automorphism of an $\nn$-oriented quadratic algebra over $S$ is the identity. Therefore, given an $\nn$-orientation $\theta$ of $\cc_0$, and $\mm$, $\mm'$ two $\cc_0$-modules, the triples $(\cc_0,\theta,\mm)$ and $(\cc_0, \theta, \mm')$ are isomorphic in the sense of Theorem~\ref{Woodgltw} if and only if $\mm$ and $\mm'$ are isomorphic as $\cc_0$-modules. Thus, to recover the Picard group of $\cc_0$, we need to fix an orientation $\theta_0$ and select all isomorphism classes having a representative of the form $(\cc_0,\theta_0, \mm)$ for some $\mm$. The possible obstructions are the existence of different quadratic algebras which are not isomorphic, and there may be different orientations of $\cc_0$ giving non-isomorphic oriented quadratic algebras. In this Subsection, we adapt the discriminant and the parity to that context, in order to overcome these problems. Notice that we do not need yet the hypothesis that $2$ is not a zero divisor on $S$. From now on, the only classes of quadratic forms we shall consider will be $\gltw$-classes.

The following Example shows that considering isomorphism classes of discriminants is too flexible.

\begin{ex} \label{exdelta}
	Let $R$ be a ring, let ${\varepsilon \in R^{\times}}$, and let ${C = R[\tau]/\left\langle \tau^2 + r\tau + s \right\rangle}$ a free quadratic algebra over $R$, where $r,s \in R$. Let $(C, \theta)$ and $(C, \theta_{\varepsilon})$ be two differently $R$-oriented copies of $C$, where ${\theta \colon \ortau \mapsto 1}$ and ${\theta_{\varepsilon} \colon \varepsilon \ortau \mapsto 1}$, with $\ortau$ the image of $\tau$ in $C/R$. Let $I$ be some invertible ideal of $C$, and assume that $I$ is free of rank $2$ over $\os$. If $q = [a,b,c]$ is a quadratic form attached to $(C,\theta,I)$, then one can compute that ${q_{\varepsilon} := [\varepsilon a, \varepsilon b, \varepsilon c]}$ corresponds to $(C,\theta_{\varepsilon},I)$, following Wood's local construction \cite[Remark~2.2]{Wood}. If we denote by ${\delta = b^2-4ac}$ the ``natural'' discriminant of $q$, then $\varepsilon^2\delta$ is the one of $q_{\varepsilon}$, hence they have isomorphic discriminant by Remark~\ref{pbdisc}. But if ${\varepsilon^2 \neq 1}$, $q$ and $q_{\varepsilon}$ are not in the same $\gltw$-class since the twisted action completely fixes the discriminant, in view of Remark~\ref{fix}.
\end{ex}

A way to choose exactly one $\gltw$-class among those of $q$ and $q_{\varepsilon}$ is to formalize this notion of ``natural'' discriminant, that is, to consider a well-chosen representative of the isomorphism class of the discriminant. The same applies for the parity.

\begin{defi}
	Let $\nn \in \pic(S)$.
	
	The \textit{oriented discriminant} of an $\nn$-oriented quadratic algebra $(\cc,\theta)$ is the global section of $\nn^{\otimes 2}$ defined by
	$$\delta_{or}(\cc,\theta) := ((\theta^{\vee})^{-1})^{\otimes 2}(\Delta(\cc)).$$
	
	The \textit{natural discriminant} of an $\nn$-twisted binary quadratic form $(\vv,q)$ is the global section of $\nn^{\otimes 2}$ defined by
	$$\delta_{nat}(\vv,q) := (\tr\otimes id_{\nn})^{\otimes 2}(\Delta(\vv,\Lambda^2\vv^{\vee}\otimes_{\os}\nn,q)),$$
	where $id_{\nn}$ is the identity map on $\nn$ and $\tr$ is the trace, that is, the canonical map ${\Lambda^2\vv \otimes_{\os} \Lambda^2\vv^{\vee} \overset{\sim}{\longrightarrow} \os}$ defined by $(x \wedge y) \otimes (\lambda \wedge \mu) \mapsto \lambda(x)\mu(y) - \lambda(y)\mu(x)$.
	
	Likewise, the \textit{oriented parity} of $(\cc,\theta)$ is the global section of $\nn/2\nn$ defined by
	$$\pi_{or}(\cc,\theta) := \overline{(\theta^{\vee})^{-1}}(\Pi(\cc)),$$
	where $\overline{(\theta^{\vee})^{-1}}$ is the map induced by $(\theta^{\vee})^{-1}$ modulo $2$.
	
	The \textit{natural parity} of $(\vv,q)$ is the global section of $\nn/2\nn$ defined by $$\pi_{nat}(\vv,q) := \overline{(\tr\otimes id_{\nn})}(\Pi(\vv,\Lambda^2\vv^{\vee}\otimes_{\os}\nn,q)).$$
\end{defi}

\begin{rem}
	We can describe all those discriminants and parities locally. Let $U \subseteq S$ be an open subset such that $\ccu$, $\vvu$, $\nnu$ are free. Let $(1,\tau)$, $(x,y)$ and $(\nu)$ be $\ou$-bases of $\ccu$, $\vvu$ and $\nnu$ respectively, let $\ortau$ be the image of $\tau$ in $(\cc/\os)_{\vert U}$, and write ${\tau^2+r\tau+s = 0}$ for some ${r,s \in \Gamma(U,\os)}$. Then, we have
	\begin{align*}
		& \delta_{or \vert U}(\cc,\theta) = (r^2-4s)\theta(\ortau)(\nu)^{-2}\nu \otimes \nu & \hspace{1cm} & \pi_{or \vert U}(\cc,\theta) \equiv r\theta(\ortau)(\nu)^{-1} \nu \pmod{2\Gamma(U,\nn)} \\
		& \delta_{nat \vert U}(\vv,q) = (b^2-4ac)\nu \otimes \nu & & \pi_{nat \vert U}(\vv,q) \equiv b \nu \pmod{2\Gamma(U,\nn)}.
	\end{align*}
\end{rem}

\begin{rem}
	In the free case, when ${\nn = \os}$, we can view directly the oriented and natural discriminants and parities as global sections of $\os$ and $\os/2\os$, expressing the previous quantities in the canonical basis of $\os$.
\end{rem}

\begin{defi}
	Let $\nn \in \pic(S)$.
	
	The \textit{oriented type} of an $\nn$-oriented quadratic algebra $(\cc,\theta)$ is the triple $(\delta_{or}(\cc,\theta),\pi_{or}(\cc,\theta),\nn)$. Equivalently, it is the triple obtained from the type of $\cc$ through the map $(\theta^{\vee})^{-1}$.
	
	Similarly, the \textit{natural type} of an $\nn$-twisted binary quadratic form $(\vv,q)$ is the triple $(\delta_{nat}(\vv,q),\pi_{nat}(\vv,q),\nn)$, also obtained from the type of $q$ through the map $\tr \otimes id_{\nn}$.
\end{defi}

Once again, the oriented parity is necessary in the general case to characterize oriented quadratic algebras. Indeed, in the same spirit as in Example~\ref{cexpar}, we can construct non-isomorphic oriented quadratic algebras $(\cc_1,\theta_1)$ and $(\cc_2,\theta_2)$ with the same oriented discriminant. But we can also find two different orientations of the same quadratic algebra giving two non-isomorphic oriented quadratic algebras with the same oriented discriminant!

\begin{ex} \label{exparor}
	Let ${R = \Z_2\!\left[2^{\frac{3}{4}}\right]}$, an order in the totally ramified extension $\Z_2[\sqrt[4]{2}]$. Observe that ${\sqrt{2} \notin R}$, but ${2\sqrt{2} = \left(2^{\frac{3}{4}}\right)^2\in R}$, so we can form the quadratic algebra
	$$C = \faktor{R[X]}{\left\langle (X - \sqrt{2})^2 \right\rangle} = \faktor{R[X]}{\left\langle X^2 - 2\sqrt{2}X + 2 \right\rangle}.$$
	
	Let ${\theta_1 \colon \overline{X} \mapsto 1+2^{\frac{3}{4}}}$ and ${\theta_2 \colon \overline{X} \mapsto 1}$ be two orientations of $C$, with $\overline{X}$ the image of $X$ in $C/R$. Notice that ${1+2^{\frac{3}{4}}}$ is indeed a unit in $R$ since ${\left(1+2^{\frac{3}{4}}\right)\left(1-2^{\frac{3}{4}}\right)\left(1-2\sqrt{2}\right) = -7}$. The oriented discriminant of $(C,\theta_i)$ is zero in both cases, but there is no isomorphism of oriented quadratic algebras between $(C,\theta_1)$ and $(C,\theta_2)$. The reason is that such an isomorphism $\psi$ should satisfy ${\psi(X) = \left(1+2^{\frac{3}{4}}\right)X + \alpha}$, for some ${\alpha \in R}$ to be determined, since we must have ${\theta_1 = \theta_2 \circ \overline{\psi}}$. Since $\psi(X)$ must be a root of the polynomial ${X^2-2\sqrt{2}X+2}$, one computes (using Equation~(\ref{eq}) for instance) that in particular, one should have ${\alpha = 2^{\frac{5}{4}}}$. This is impossible, since $2^{\frac{5}{4}}$ does not belong to $R$.
\end{ex}

\begin{rem}
	Instead of ${\Z_2\!\left[2^{\frac{3}{4}}\right]}$, one can take ${R = \Z\!\left[\dfrac{1}{7},2^{\frac{3}{4}}\right]}$ in Example~\ref{exparor}.
\end{rem}

\begin{ex}
	Let ${R = \faktor{\Z[X,Y]}{\left\langle X^2-8, Y^2-8 \right\rangle}}$ and ${C = \faktor{R[Z]}{\left\langle Z^2 + XZ + 2 \right\rangle}}$. Let ${\theta_1 \colon \overline{Z} \mapsto 1}$ and ${\theta_2 \colon \overline{Z} \mapsto 3-Y}$ be two orientations (where ${3-Y \in R^{\times}}$ since ${(3-Y)(3+Y) = 1}$). The oriented discriminant of $(C,\theta_i)$ is $0$ in both cases. Following the reasoning of Example~\ref{exparor}, we show that there is no isomorphism of oriented quadratic algebras between $(C,\theta_1)$ and $(C,\theta_2)$. Indeed, any isomorphism $\psi$ should satisfy ${\psi(Z) = (3+Y)Z + \alpha}$ for some ${\alpha \in R}$, and the fact that $\psi(Z)$ is a root of the polynomial ${Z^2+XZ+2}$ implies that ${2\alpha = 2X + XY}$, but $2$ does not divide $XY$ in $R$, a contradiction.
\end{ex}

\begin{prop} \label{natortype}
	Let $\nn \in \pic(S)$. Two $\nn$-twisted binary quadratic forms which are equivalent have equal natural type. Likewise, two isomorphic $\nn$-oriented quadratic algebras have equal oriented type. Furthermore, Wood's bijection (Theorem~\ref{Woodgltw}) is type-preserving, meaning that if a twisted quadratic form $(\vv,q)$ matches to a triple $(\cc,\theta,\mm)$, then ${\delta_{nat}(\vv,q) = \delta_{or}(\cc,\theta)}$ and ${\pi_{nat}(\vv,q) = \pi_{or}(\cc,\theta)}$.
\end{prop}

\begin{proof}
	Let $(\vv,q)$ and $(\vv',q')$ be two equivalent $\nn$-twisted binary quadratic forms: there exists an isomorphism of $\os$-modules ${f \colon \vv \overset{\sim}{\longrightarrow} \vv'}$ sending $q$ to $q'$. Locally on some open subset ${U \subseteq S}$ such that $\vvu$ and $\nnu$ are free, let $(x,y)$ and $(\nu)$ be $\ou$-bases of $\vvu$ and $\nnu$ respectively, and set ${(x',y') := (f_{\vert U}(x),f_{\vert U}(y))}$ for an $\ou$-basis of $\vv'_{\vert U}$. Denote by ${q = ax^2z + bxyz + cy^2z}$ where ${z = \nu\det_{(x,y)}}$; then $f_{\vert U}$ sends $q$ to ${q' = ax'^2z' + bx'y'z'+cy'^2z'}$. The coefficients being unmodified, the computation of the natural discriminant and parity is everywhere locally the same.
	
	For $\nn$-oriented quadratic algebras, the reasoning is similar. It remains to check the correspondence. According to Remark 2.2 in \cite{Wood}, locally on ${U \subseteq S}$ where $\vvu$ and $\nnu$ are free, if ${q = ax^2z+bxyz+cy^2z}$ in some $\ou$-basis $(x,y)$ of $\vvu$ (and ${z = \nu\det_{(x,y)}}$), then the corresponding quadratic algebra $\ccu$ is given by the $\ou$-basis $(1,\tau)$ where ${\tau^2 + b\tau + ac = 0}$, and its orientation $\theta$ sends $\ortau$ to $\nu^{\vee}$. On the one hand, we have ${\delta_{nat \vert U}(\vv, q) = (b^2-4ac)\nu\otimes\nu}$; on the other hand, ${\delta_{or \vert U}(\cc,\theta) = (b^2-4ac)\theta(\ortau)(\nu)^{-2}\nu\otimes\nu = (b^2-4ac)\nu\otimes\nu = \delta_{nat \vert U}(\vv,q)}$, as desired. Clearly, we also have ${\pi_{nat \vert U}(\vv,q) \equiv b\nu \equiv \pi_{or \vert U}(\cc,\theta)\pmod{2\Gamma(U,\nn)}}$, hence $(\vv,q)$ and $(\cc,\theta)$ have equal types.
\end{proof}

\subsection{Recovering the Picard group from Wood's twisted quadratic forms} \label{sectwpic}

Under the condition that $2$ is not a zero divisor on our base scheme $S$, it is enough to fix the oriented discriminant and the oriented parity in order to recover the Picard group in Theorem~\ref{Woodgltw}. This extra condition on $2$ is necessary, as shown in Examples~\ref{pasuni} and~\ref{pasuni2}.

\begin{theo} \label{oruniqueness}
	Let $S$ be a scheme, and assume that $2$ is not a zero divisor on $S$. Then we have the following bijection:
	\begin{align*}
		\left\{
		\begin{array}{c}
		valid~triples~ (\delta,\pi,\nn)
		\end{array}
		\right\} & \overset{1 : 1}{\longleftrightarrow} \left\{
		\begin{array}{c}
		isomorphism~classes~of \\
		\nn \mhyphen oriented~quadratic~algebras
		\end{array}
		\right\} \\
		(\delta,\pi,\nn) ~ & \longmapsto ~\Big[\left(\Omega(\delta,\pi,\nn), id_{\nn^{\vee}}\right)\Big] \\
		(\delta_{or}(\cc,\theta),\pi_{or}(\cc,\theta),\nn) ~ & \longmapsfrom ~ \Big[(\cc,\theta)\Big]
	\end{align*}
	where $\Omega(\delta,\pi,\nn)$ is the quadratic algebra constructed in Proposition~\ref{constr}, and $id_{\nn^{\vee}}$ denotes the identity map ${\Omega/\os = \nn^{\vee} \longrightarrow \nn^{\vee}}$.
\end{theo}

\begin{proof}
	We proceed as for Theorem~\ref{uniqueness}. The map sending the isomorphism class of an $\nn$-oriented quadratic algebra to its oriented type is well-defined by Proposition~\ref{natortype}, and surjective by definition of valid triples. For injectivity, we check that every $\nn$-oriented quadratic algebra $(\cc,\theta)$ of oriented type $(\delta,\pi,\nn)$ is isomorphic to $(\Omega(\delta,\pi,\nn),id_{\nn^{\vee}})$. For this part, the proof is the same as for Theorem~\ref{uniqueness}, except that $\theta$ is not chosen but given. The commutativity of the local isomorphisms $\Psi_i$'s with the orientations $\theta$ and $id_{\nn^{\vee}}$ is a straightforward computation.
\end{proof}

Finally, we can prove Theorem~\ref{goal}, which we make precise here:

\begin{theo}\label{bij}
	Let $S$ be a scheme in which $2$ is not a zero divisor. Let $\cc_0$ be a quadratic algebra over $S$. Let $\nn$ be a representative of the isomorphism class of the $\os$-module $(\cc_0/\os)^{\vee}$, and fix an isomorphism of $\os$-modules ${\theta \colon \cc/\os \overset{\sim}{\longrightarrow} \nn^{\vee}}$. Set $\delta := \delta_{or}(\cc,\theta)$ and $\pi := \pi_{or}(\cc,\theta)$.
	
	Then Wood's bijection (Theorem~\ref{Woodgltw}) restricts to a bijection
	$$ \left\{
	\begin{array}{c}
	\gltw \!{} \mhyphen equivalence~classes~of~primitive \\
	\nn \mhyphen twisted~binary~quadratic~forms \\
	with~natural~type~(\delta,\pi,\nn)
	\end{array}
	\right\} \overset{1 : 1}{\longleftrightarrow} \pic(\cc_0),$$
	which endows the set on the left-hand side with a group structure.
\end{theo}

\begin{proof}
	As seen in Theorem~\ref{oruniqueness}, there is a unique isomorphism class of $\nn$-oriented quadratic algebras with oriented discriminant ${\delta_{or} = \delta}$ and oriented parity $\pi_{or} = \pi$. Recall that by Proposition~\ref{auto}, there is no non-trivial automorphism of $(\cc_0,\theta)$, hence two $\cc_0$-modules $\mm$ and $\mm'$ are equivalent in Wood's sense (\textit{cf}. Theorem~\ref{Woodgltw}) if and only if they are isomorphic as $\cc_0$-modules. Therefore, once a representative $(\cc_0,\theta)$ is fixed, the bijection restricts to isomorphism classes of invertible $\cc_0$-modules. Since the bijection is type-preserving (Proposition~\ref{natortype}), this corresponds to selecting, on the other side of the bijection, the $\gltw$-classes of primitive $\nn$-twisted quadratic forms having natural discriminant ${\delta_{nat} = \delta}$ and natural parity $\pi_{nat} = \pi$.
\end{proof}

In particular, Wood's local description of the bijection still holds in Theorem~\ref{bij}. It enables us to deduce the following special case:

\begin{cor} \label{expl}
	Let $R$ be an integral domain of characteristic different from $2$, and assume that every locally free $R$-module of finite rank is free. Let ${\pi \in R/2R}$, and let ${\tilde{\pi} \in R}$ be a lift of $\pi$. Let ${\delta \in R}$ be such that ${\delta \equiv \tilde{\pi}^2 \pmod{4R}}$. Then we have the global bijection
	\begin{align*}
		\left\{
		\begin{array}{c}
		\gltw \!{} \mhyphen classes~of~primitive \\
		twisted~binary~quadratic~forms \\
		with~natural~discriminant~\delta \\
		and~natural~parity~\pi
		\end{array}
		\right\} & \overset{1 \colon 1}{\longleftrightarrow} \pic\left(\faktor{R[\omega]}{\left\langle \omega^2 + \tilde{\pi} \omega - \frac{\delta-\tilde{\pi}^2}{4} \right\rangle}\right) \\
		\bigg[[a,b,c] \text{~with~} a\neq 0\bigg] & \longmapsto \bigg[\langle \omega + \frac{\tilde{\pi}-b}{2}, a \rangle\bigg]
	\end{align*}
	where the notation $[a,b,c]$ means that we have chosen an $R$-basis $(x,y)$ for our twisted binary quadratic form in which it can be expressed as ${ax^2z+bxyz+cy^2z}$ (where ${z = \det_{(x,y)}}$).
\end{cor}

\begin{rem}
	We have ${[a,b,c] \sim [c,-b,a]}$ and ${[a,b,c] \sim [a+b+c,b+2c,c]}$ through the matrices $\begin{pmatrix} 0 & -1 \\ 1 & 0 \end{pmatrix}$ and $\begin{pmatrix} 1 & 0 \\ 1 & 1 \end{pmatrix}$. Since $[a,b,c]$ is primitive, the quantities $a$, $c$ and ${a+b+c}$ cannot vanish at the same time, hence every primitive quadratic form over $R$ is equivalent to one whose first coefficient is nonzero.
\end{rem}

\begin{proof}
	We follow Wood's local description of the bijection \cite[Remark~2.2]{Wood}. Since every locally free $R$-module of finite rank is free by assumption, every twisted binary quadratic form is globally defined by an equation of the form ${ax^2z+bxyz+cy^2z}$ with ${z = \det_{(x,y)}}$.
	
	Given such a quadratic form $[a,b,c]$ with ${a \neq 0}$, let ${C = R[\tau]/\left\langle \tau^2+b\tau + ac \right\rangle}$ and $M$ be the associated quadratic algebra and invertible $C$-module. There exist ${x,y \in M}$ such that ${M = xR \oplus yR}$ as $R$-modules, and ${\tau \cdot x = -by-cx}$, ${\tau \cdot y = ax}$. It is a straightforward computation to check that the homomorphism of $R$-modules
	$$\begin{array}{rcl}
		\varphi ~\colon M & \longrightarrow & \left\langle \tau, a \right\rangle \\
		x & \mapsto & \tau \\
		y & \mapsto & a
	\end{array}$$
	is a surjective homomorphism of $C$-modules. Since $R$ is an integral domain, and since ${a \neq 0}$, the ideal ${\left\langle \tau, a \right\rangle}$ has rank $2$ as an $R$-module, hence $\varphi$ is injective. Thus, $\varphi$ is an isomorphism of $C$-modules, and $[a,b,c]$ corresponds to the ideal ${\left\langle \tau, a \right\rangle}$ in $\pic(C)$.
	
	The isomorphism ${R[\tau]/\langle \tau^2+b\tau+ac \rangle \overset{\sim}{\longrightarrow} R[\omega]/\langle \omega^2+\tilde{\pi} \omega - \frac{\delta-\tilde{\pi}^2}{4} \rangle}$ sending $\tau$ to ${\omega + \frac{\tilde{\pi}-b}{2}}$ enables to conclude.
\end{proof}

Some examples of rings satisfying all hypotheses of Corollary~\ref{expl} are $R$ and $R[X]$ for every principal ideal domain $R$ such that ${\charac(R) \neq 2}$ \cite{Seshadri}.

\end{sloppypar}

 \providecommand{\andname}{\&} \providecommand{\bibsep}{0cm}
\providecommand{\toappear}{to appear} \providecommand{\noopsort}[1]{}

\textsc{William Dallaporta}, Institut de Mathématiques de Toulouse, Université de Toulouse, CNRS UMR 5219, 118 route de Narbonne, 31062 Toulouse Cedex 9, France.

E-mail address: \url{william.dallaporta@laposte.net}

\end{document}